\documentclass[12pt,a4paper,english,reqno]{amsart}
\usepackage{geometry}
\usepackage{amsmath,amssymb,amsthm,mathtools}
\usepackage[mathscr]{euscript}
\usepackage{tikz,babel,enumitem,adjustbox}
\usepackage[final]{microtype}
\usepackage[numbers]{natbib}
\usetikzlibrary{decorations.pathreplacing,calc}
\usepackage{hyperref}
\hypersetup{colorlinks=true,linkcolor=blue,citecolor=blue,pdfpagemode=UseNone,pdfstartview={XYZ null null 1.00}}
\usepackage{cmtiup}

\setenumerate{label=(\alph*)}

\theoremstyle{plain}
\newtheorem{theorem}{Theorem}[section]
\newtheorem{lemma}[theorem]{Lemma}
\newtheorem{claim}[theorem]{Claim}
\newtheorem{proposition}[theorem]{Proposition}
\newtheorem{corollary}[theorem]{Corollary}
\newtheorem{conjecture}[theorem]{Conjecture}
\newtheorem{problem}[theorem]{Problem}
\theoremstyle{remark}

\newenvironment{poc}{\begin{proof}[Proof of claim]}{\end{proof}}

\def\TT{\mathcal{T}_{9}}
\def\PP{\mathcal{P}_{12}}
\def\CC{{C}}
\def\N{\mathbb{N}}
\def\Z{\mathbb{Z}}
\def\mindeg{\delta}
\def\codeg{\delta_2}
\def\ss{\mathbb{S}}
\def\SS{\mathscr{S}}
\def\HH{\mathcal{H}}
\def\GG{\mathcal{G}}

\DeclareMathOperator\Deg{d}
\let\emptyset\varnothing
\newcommand*{\abs}[1]{\lvert#1\rvert}
\newcommand{\kzero}{33}

\newcommand{\e}{\ensuremath{\varepsilon}}
\newcommand{\al}{\alpha}

\let\originalleft\left
\let\originalright\right
\renewcommand{\left}{\mathopen{}\mathclose\bgroup\originalleft}
\renewcommand{\right}{\aftergroup\egroup\originalright}

\makeatletter
\def\imod#1{\allowbreak\mkern10mu({\operator@font mod}\,\,#1)}
\makeatother

\begin{document}
	
	\title{Spanning surfaces in 3-graphs}
	
	\author{Agelos Georgakopoulos}
	\address{Mathematics Institute, University of Warwick, CV4\thinspace7AL, UK}
	\urladdr{http://homepages.warwick.ac.uk/~maslar/}
	
	\author{John Haslegrave}
	\address{Mathematics Institute, University of Warwick, CV4\thinspace7AL, UK}
	\email{j.haslegrave@cantab.net}
	
	\author{Richard Montgomery}
	\address{School of Mathematics, University of Birmingham, B15\thinspace2TT, UK}
	\email{r.h.montgomery@bham.ac.uk}
	
	\author{Bhargav Narayanan}
	\address{Department of Mathematics,	Rutgers University, Piscataway NJ 08854, USA}
	\email{narayanan@math.rutgers.edu}
	
	\date{17 August 2018}
	
	\begin{abstract}
		We prove a topological extension of Dirac's theorem suggested by Gowers in 2005: for any connected, closed surface $\mathscr{S}$, we show that any two-dimensional simplicial complex on $n$ vertices in which each pair of vertices belongs to at least $n/3 + o(n)$ facets contains a homeomorph of $\mathscr{S}$ spanning all the vertices. This result is asymptotically sharp, and implies in particular that any 3-uniform hypergraph on $n$ vertices with minimum codegree exceeding $n/3+o(n)$ contains a spanning triangulation of the $2$-sphere.
	\end{abstract}
	
	\maketitle
	
	\section{Introduction}
	In this paper, we extend the classical graph-theoretic result of Dirac~\citep{dirac} on spanning cycles to the setting of simplicial $2$-complexes, or, equivalently, the setting of $3$-uniform hypergraphs (or \emph{$3$-graphs} for short). Dirac's theorem determines the best-possible minimum degree condition which guarantees that an $n$-vertex graph contains a \emph{Hamiltonian cycle}, i.e., a cycle spanning the entire vertex set of the graph. A natural generalisation is to treat a `spanning cycle in a $3$-graph' as a triangulation of the $2$-sphere spanning the vertex set, and here we determine asymptotically the best-possible minimum codegree condition which guarantees the existence of such an object in an $n$-vertex $3$-graph.
	
	Dirac's theorem is of central importance in graph theory, and a number of different extensions to $3$-graphs have previously been shown; see~\citep{icm} for a broad overview. All these results --- see~\citep{bermond, katona, rodl-1, rodl-3}, for example --- treat a `spanning cycle in a $3$-graph' as a rigid pattern of interlocking edges with respect to some cyclic ordering of the underlying vertex set, an inherently one-dimensional notion. Here, we shall take a topological point of view, and consider a two-dimensional extension to $3$-graphs.
	
	To motivate our point of view, already implicit in the work of Brown, Erd\H{o}s and S\'os~\citep{erdos}, we start by observing that a Hamiltonian cycle in a graph $G$ is a set of edges of $G$ such that the simplicial complex induced by these edges is homeomorphic to $\ss^1$, the one-dimensional sphere, where, additionally, the $0$-skeleton of this complex is the entire vertex set of $G$. By analogy, we define \emph{a copy of the sphere} in a $3$-graph $\HH$ to be a set of edges of $\HH$ such that the simplicial complex induced by these edges is homeomorphic to $\ss^2$, the two-dimensional sphere, and we say that a copy of the sphere in a $3$-graph $\HH$ is \emph{spanning} or \emph{Hamiltonian} if the $0$-skeleton of the associated simplicial complex is the entire vertex set of $\HH$. The following natural question, in the spirit of Dirac's theorem, was raised independently by Gowers~\citep{tim1} and Conlon~\citep{conlon}.
	
	\begin{problem}\label{q:spheredeg}
		What degree conditions guarantee the existence of a spanning copy of the sphere in a $3$-graph?
	\end{problem}
	
	Equivalently, Problem~\ref{q:spheredeg} asks for degree conditions that guarantee the existence of a homeomorphic copy of $\ss^2$ containing all the vertices in a simplicial 2-complex. We will, however, adhere to the language of hypergraphs in what follows, as this is the language in which most related results have been formulated.
	
	Before we turn to answering the above question, let us place the problem in a more general context. While the circle $\ss^1$ is, up to homeomorphism, the unique connected, closed $1$-manifold, this is no longer the case in two dimensions. Hence, one can ask a more general question by replacing the sphere with an arbitrary connected, closed $2$-manifold (or \emph{surface} for short). The following more general question fits into the `higher-dimensional combinatorics' programme of Linial~\citep{linial1, linial2, linial3}, and was also suggested by Gowers~\citep{tim2}.
	
	\begin{problem}\label{q:surfacedeg}
		Given a surface $\SS$, what degree conditions guarantee the existence of a spanning copy of $\SS$ in a $3$-graph?
	\end{problem}
	
	Of course, by the classification theorem (see~\citep{toptext, zip}, for example), every surface is homeomorphic to either the sphere $\ss^2$, a connected sum of finitely many tori, or a connected sum  of finitely many real projective planes. Also, to be clear, a copy of a surface $\SS$ in a $3$-graph $\HH$ is, as before, a set of edges of $\HH$ such that the simplicial complex induced by these edges is homeomorphic to $\SS$, and we say that a copy of $\SS$ in a $3$-graph $\HH$ is spanning if the $0$-skeleton of the associated simplicial complex is the entire vertex set of $\HH$.
	
	Finally, instead of asking for a spanning copy of a specific surface, one might settle for a spanning copy of any surface whatsoever. In other words, we could ask the following weaker question.
	
	\begin{problem}\label{q:timques}
		What degree conditions guarantee the existence of a spanning copy of some surface in a $3$-graph?
	\end{problem}
	
	In this paper, we shall prove a result that gives an asymptotically sharp solution to Problems~\ref{q:spheredeg},~\ref{q:surfacedeg} and~\ref{q:timques}. Recall that the codegree of a pair of vertices in a $3$-graph $\HH$ is the number of edges of $\HH$ containing the vertex pair. Writing $\codeg(\HH)$ for the minimum codegree of a $3$-graph $\HH$, our main result says the following.
	
	\begin{theorem}\label{mainthm}
		For every surface $\SS$ and every $\mu > 0$, the following holds for all sufficiently large $n\in \N$: any $3$-graph $\HH$ on $n$ vertices with $\codeg(\HH) \ge n/3 + \mu n$ contains a spanning copy of $\SS$.
		
		Moreover, for each $n \in \N$, there exists a $3$-graph $\HH$ on $n$ vertices with $\codeg(\HH) = \lfloor n/3 \rfloor - 1$ such that there are at most $2\lceil n/3 \rceil$ vertices in the $0$-skeleton of a copy of any surface in $\HH$.
	\end{theorem}
	
	\begin{figure} \label{figXYZ}
		\begin{center}
			
			\begin{tikzpicture}[scale = 0.64]
			
			\draw (0,0) circle (2cm);
			\draw (6,0) circle (2cm);
			\draw (12,0) circle (2cm);
			\draw (18,0) circle (2cm);
			\node at (0, 0) [] {$X$};
			\node at (6, 0) [] {$Y$};
			\node at (12, 0) [] {$Z$};
			\node at (18, 0) [] {$X$};

			\node at (0.8, 0.8) [inner sep=0.7mm, circle, fill=black!100] {};
			\node at (17.2, -0.8) [inner sep=0.7mm, circle, fill=black!100] {};
			\node at (0.8, -0.8) [inner sep=0.7mm, circle, fill=black!100] {};
			\node at (5.2, -0.8) [inner sep=0.7mm, circle, fill=black!100] {};
			\node at (6.8, +0.8) [inner sep=0.7mm, circle, fill=black!100] {};
			\node at (6.8, -0.8) [inner sep=0.7mm, circle, fill=black!100] {};
			\node at (6+5.2, -0.8) [inner sep=0.7mm, circle, fill=black!100] {};
			\node at (6+6.8, +0.8) [inner sep=0.7mm, circle, fill=black!100] {};
			\node at (6+6.8, -0.8) [inner sep=0.7mm, circle, fill=black!100] {};

			\draw [thick] (0.8, 0.8) -- (0.8, -0.8) -- (5.2, -0.8) -- (0.8, 0.8);
			\draw [thick] (6.8, 0.8) -- (6.8, -0.8) -- (11.2, -0.8) -- (6.8, 0.8);
			\draw [thick] (12.8, 0.8) -- (12.8, -0.8) -- (17.2, -0.8) -- (12.8, 0.8);
			
			\end{tikzpicture}
		\end{center}
		\caption{An extremal construction.}\label{fig:tripart}
	\end{figure}
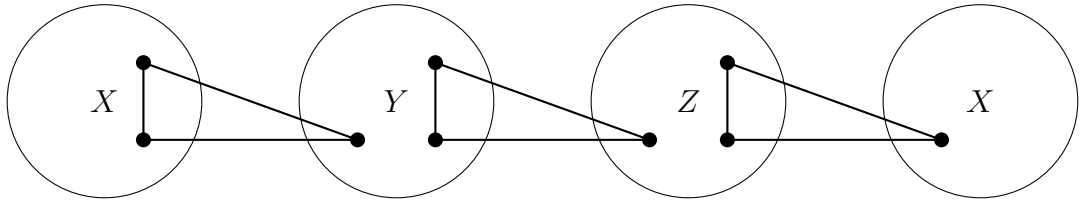
	
	The second half of Theorem~\ref{mainthm} follows from the simple construction shown in Figure~\ref{fig:tripart}. It will be helpful to have some notation to discuss this construction: we define a relation on the edge set of a $3$-graph $\HH$ by saying that that two edges of $\HH$ \emph{touch} if they intersect in two vertices, and we call an equivalence class of edges in the transitive closure of this relation a \emph{tight component} of $\HH$. Observe that the set of edges constituting a copy of a surface in a $3$-graph $\HH$ must belong to a single tight component of $\HH$ since all the surfaces under consideration are without boundary. Now, given $n \in \N$, let $X$, $Y$ and $Z$ be three disjoint sets of vertices, with sizes as equal as possible, such that $|X| + |Y| + |Z| = n$ and consider, as in Figure~\ref{fig:tripart}, the $3$-graph $\HH$ on the vertex set $X \cup Y \cup Z$ whose edge set consists of all triples either intersecting $X$ in two vertices and $Y$ in one, intersecting $Y$ in two vertices and $Z$ in one, or intersecting $Z$ in two vertices and $X$ in one. It is easy to see that $\codeg(\HH) = \lfloor n/3 \rfloor - 1$. Furthermore, it is clear that the edge set of $\HH$ consists of three tight components, with each tight component spanning two of the three vertex classes $X$, $Y$ and $Z$. Thus, there are at most $2\lceil n/3 \rceil$ vertices in the $0$-skeleton of a copy of any surface in $\HH$.
	
	It is left then to prove the first half of Theorem~\ref{mainthm}. One main obstacle to doing this, in constrast to Dirac's theorem, is that, potentially, many of the edges inherently cannot contribute to a spanning copy of any surface. The construction above shows that for each $n \in \N$, there exists a $3$-graph $\HH$ with $\codeg(\HH) = \lfloor n/3 \rfloor - 1$ which does not contain a spanning tight component. As we shall see (in Proposition~\ref{toycomp}), a $3$-graph $\HH$ on $n$ vertices whose minimum codegree exceeds $n/3$ contains a spanning tight component. However, it may additionally have another non-spanning tight component whose edges are of no use whatsoever when trying to build a spanning copy of any surface. This already presents a challenge, but the situation is in fact more intricate: indeed, as we shall see (in Conjecture~\ref{tightcompconj} and the discussion preceding it), even assuming that $\HH$ has a unique spanning tight component does not make the problem at hand significantly easier.
	
	It is perhaps worth mentioning that the construction depicted in Figure~\ref{fig:tripart} is not uniquely extremal, and indeed, there exist other non-isomorphic families of constructions that demonstrate that Theorem~\ref{mainthm} is tight. Given a surface $\SS$ with Euler characteristic $\chi \in \Z$, and any sufficiently large natural number $n \in \N$, we construct a $3$-graph on $n$ vertices with no spanning copy of $\SS$ as follows. Let $X$ and $Y$ be two disjoint sets of vertices, with $|X| + |Y| = n$ and $|X|$ the least integer exceeding $(2n-2\chi)/3$, and let $\HH$ be the $3$-graph on $X \cup Y$ whose edge set consists of all triples meeting $X$ in an odd number of vertices. It is clear that $\codeg(\HH) = |Y| - 1 = \lceil (n + 2\chi)/3 \rceil - 2$, and it is easily verified that this construction is not isomorphic to the one discussed earlier. Now, suppose for a contradiction that there is a spanning copy of $\SS$ in $\HH$. This copy of $\SS$ cannot contain any triple contained entirely in $X$ since these edges all lie in a single tight component that spans $X$ but not $Y$. Next, view each edge of $\HH$ in this copy of $\SS$ as a facet of a triangulation of $\SS$ and count, for each facet, the number of vertices of $X$ on its boundary. Each facet contains exactly one vertex in $X$, so this quantity is equal to the number of facets, which by Euler's formula, is $2n - 2\chi$. On the other hand, each vertex in $X$ is on the boundary of at least three facets since $\SS$ has no boundary, so this quantity is at least $3|X|$. We conclude that $2n - 2\chi \ge 3|X|$, which is a contradiction.
	
	A few comments about results in the vicinity of Theorem~\ref{mainthm} are also in order. Questions in the spirit of Problem~\ref{q:spheredeg} have previously been asked for graphs: K\"uhn, Osthus and Taraz~\citep{kot} proved a Dirac-type theorem for finding spanning planar triangulations in graphs; applying Theorem~\ref{mainthm} to the $3$-graph of all the triangles in a graph recovers their result. In the probabilistic setting, it is natural to ask when a spanning copy of the sphere is likely to appear in the binomial random $3$-graph: Luria and Tessler \citep{threshold} recently established a sharp threshold result for this problem. Finally, it has also been brought to our attention that Conlon, Ellis and Keevash~\citep{conlon} earlier proved (in unpublished work) a weaker statement in the direction of Theorem~\ref{mainthm}, showing that any $3$-graph $\HH$ on $n$ vertices with $\codeg(\HH) \ge 2n/3 + o(n)$ contains a spanning copy of the sphere.
	
	This paper is organised as follows. We establish some notation and collect together the various results required for the proof of our main result in Section~\ref{s:prelim}. We give a short sketch of the proof of Theorem~\ref{mainthm} highlighting the main obstacles and how we circumvent them in Section~\ref{s:outline}; the proof proper follows in Section~\ref{s:proof}. Finally, we conclude by discussing some open problems and directions for further research in Section~\ref{s:conc}.
	
	\section{Preliminaries}\label{s:prelim}
	For $n \in \N$, let $[n]=\{1, 2, \dots, n\}$. For a set $X$ and $r \in \N$, we write $X^{(r)}$ for the family of $r$-element subsets of $X$. In this language, an $r$-graph $G$ is a pair $(V,E)$ of finite sets with $E \subset V^{(r)}$. Here, we shall only be concerned with $2$-graphs and $3$-graphs; as usual, we refer to $2$-graphs as graphs.
	
	Let $G = (V,E)$ be a graph. For a vertex $x \in V(G)$, its \emph{neighbourhood $N_G(x)$} is the set of vertices adjacent to $x$, and its \emph{degree $\Deg_G(x)$} is the size of $N_G(x)$. The \emph{minimum degree $\mindeg(G)$} of $G$ is
	\[\mindeg(G)=\min\{\Deg_G(x): x \in V(G)\}.\]
	For any two sets of vertices $X, Y \subset V(G)$, we write $E_G(X,Y)$ for the set of edges with one endpoint in $X$ and one endpoint in $Y$.
	
	Next, let $\HH = (V,E)$ be a $3$-graph. For a pair of distinct vertices $x,y \in V(\HH)$, we define their \emph{neighbourhood $N_\HH(x,y)$} to be the set of vertices $z \in V(\HH)$ such that $x y z \in E(\HH)$, and we define the \emph{codegree $\Deg_\HH(x,y)$} of $x$ and $y$ to be size of $N_\HH(x, y)$. The \emph{minimum codegree $\codeg(\HH)$} of $\HH$ is
	\[\codeg(\HH)=\min\{\Deg_\HH(x,y): x,y \in V(\HH)  \text{ and } x \ne y\}.\]
	For any vertex $v \in V(\HH)$, we define the link graph $L_\HH(v)$ of $v$ to be the graph on $V(\HH) \setminus \{ v \}$ where two vertices $x$ and $y$ are joined with an edge if $v \in N_\HH(x,y)$. Finally, as mentioned earlier, we define a relation on the edge set of a $3$-graph $\HH$ by saying that that two edges of $\HH$ \emph{touch} if they intersect in two vertices, and we call an equivalence class of edges in the transitive closure of this relation a \emph{tight component} of $\HH$.
	
	We will use the following classical result of Erd\H{o}s~\citep{degen} generalising an old result of K\"ovari, S\'os and Tur\'an~\citep{kovari}. Recall that an $r$-graph $\GG$ is \emph{degenerate} if there exists a colouring of $V(\GG)$ with $r$ colours such that each edge of $\GG$ meets each of the $r$ colour classes.
	\begin{theorem}\label{kov}
		For each degenerate $r$-graph $\GG$, there exists $c >0$ such that every $r$-graph $\HH$ on $n$ vertices with $|E(\HH)| \ge n^{r - c}$ contains a copy of $\GG$ as a subgraph. \qed
	\end{theorem}
	
	In our proofs, we shall require the conclusion of Theorem~\ref{kov} for some specific degenerate graphs and $3$-graphs that we now define. For $k \in \N$, we denote a cycle of length $k$ by $\CC_k$; recall that if $k$ is even, then the $k$-cycle $\CC_k$ is degenerate. Next, we define two degenerate $3$-graphs as in Figure~\ref{fig:degen-col}, namely a $3$-graph $\TT$ on the vertex set $[9]$ and a $3$-graph $\PP$ on the vertex set $[12]$. It is clear that the simplicial complex induced by $\TT$ is homeomorphic to the two-dimensional torus. It may also be verified that the simplicial complex induced by $\PP$ is homeomorphic to the real projective plane; indeed, it is not hard to see that $\PP$ is obtained from a simple modification of the standard $6$-point triangulation of the real projective plane.
	
	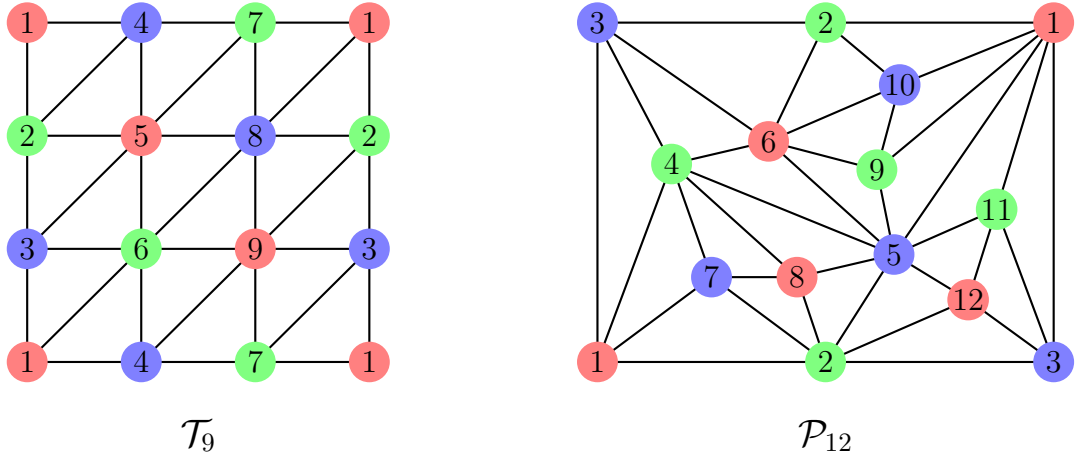
\begin{figure}
		\begin{center}
			
			\begin{tikzpicture}[scale = 0.75]
			\foreach \x in {0,2,4,6}
			\foreach \y in {0,2,4}
			\draw [thick] (\x,\y) -- (\x,\y+2);
			
			\foreach \x in {0,2,4}
			\foreach \y in {0,2,4,6}
			\draw [thick] (\x,\y) -- (\x+2,\y);
			
			\foreach \x in {0,2,4}
			\foreach \y in {0,2,4}
			\draw [thick] (\x,\y) -- (\x+2,\y+2);
			
			\node at (0, 6) [inner sep=0.7mm, circle, fill=red!50] {$1$};
			\node at (0, 4) [inner sep=0.7mm, circle, fill=green!50] {$2$};
			\node at (0, 2) [inner sep=0.7mm, circle, fill=blue!50] {$3$};
			\node at (0, 0) [inner sep=0.7mm, circle, fill=red!50] {$1$};
			
			\node at (2, 6) [inner sep=0.7mm, circle, fill=blue!50] {$4$};
			\node at (2, 4) [inner sep=0.7mm, circle, fill=red!50] {$5$};
			\node at (2, 2) [inner sep=0.7mm, circle, fill=green!50] {$6$};
			\node at (2, 0) [inner sep=0.7mm, circle, fill=blue!50] {$4$};
			
			\node at (4, 6) [inner sep=0.7mm, circle, fill=green!50] {$7$};
			\node at (4, 4) [inner sep=0.7mm, circle, fill=blue!50] {$8$};
			\node at (4, 2) [inner sep=0.7mm, circle, fill=red!50] {$9$};
			\node at (4, 0) [inner sep=0.7mm, circle, fill=green!50] {$7$};
			
			\node at (6, 6) [inner sep=0.7mm, circle, fill=red!50] {$1$};
			\node at (6, 4) [inner sep=0.7mm, circle, fill=green!50] {$2$};
			\node at (6, 2) [inner sep=0.7mm, circle, fill=blue!50] {$3$};
			\node at (6, 0) [inner sep=0.7mm, circle, fill=red!50] {$1$};
			
			\draw [thick] (10,0) -- (18,0) -- (18,6) -- (10,6) -- (10,0);
			\draw [thick] (11.3, 3.5) -- (15.2, 1.9) -- (13, 3.9) -- (11.3, 3.5);
			\draw [thick] (10,0) -- (11.3, 3.5) -- (10,6);
			\draw [thick] (10,6) -- (13, 3.9) -- (14,6);
			\draw [thick] (10,0) -- (12, 1.5) -- (13.5, 1.5) -- (15.2, 1.9 ) -- (14,0);
			\draw [thick] (12, 1.5) -- (11.3, 3.5) -- (13.5, 1.5);
			\draw [thick] (12, 1.5) -- (14, 0) -- (13.5, 1.5);
			\draw [thick] (15.2, 1.9 ) -- (18, 6);
			\draw [thick] (14,6) -- (15.3, 4.9) -- (14.9, 3.4) -- (15.2, 1.9 );
			\draw [thick] (15.3, 4.9) -- (13, 3.9) -- (14.9, 3.4);
			\draw [thick] (15.3, 4.9) -- (18, 6) -- (14.9, 3.4);
			\draw [thick] (14,0) -- (16.5, 1.1) -- (17, 2.7) -- (18, 6);
			\draw [thick] (16.5, 1.1) -- (15.2, 1.9 ) -- (17, 2.7);
			\draw [thick] (16.5, 1.1) -- (18, 0) -- (17, 2.7);
			
			\node at (10, 0) [inner sep=0.7mm, circle, fill=red!50] {$1$};
			\node at (14, 0) [inner sep=0.7mm, circle, fill=green!50] {$2$};
			\node at (18, 0) [inner sep=0.7mm, circle, fill=blue!50] {$3$};
			\node at (10, 6) [inner sep=0.7mm, circle, fill=blue!50] {$3$};
			\node at (14, 6) [inner sep=0.7mm, circle, fill=green!50] {$2$};
			\node at (18, 6) [inner sep=0.7mm, circle, fill=red!50] {$1$};
			\node at (11.3, 3.5) [inner sep=0.7mm, circle, fill=green!50] {$4$};
			\node at (15.2, 1.9 ) [inner sep=0.7mm, circle, fill=blue!50] {$5$};
			\node at (13, 3.9 ) [inner sep=0.7mm, circle, fill=red!50] {$6$};
			\node at (12, 1.5) [inner sep=0.7mm, circle, fill=blue!50] {$7$};
			\node at (13.5, 1.5) [inner sep=0.7mm, circle, fill=red!50] {$8$};
			\node at (14.9, 3.4) [inner sep=0.7mm, circle, fill=green!50] {$9$};
			\node at (15.3, 4.9) [inner sep=0.2mm, circle, fill=blue!50] {$10$};
			\node at (17, 2.7) [inner sep=0.2mm, circle, fill=green!50] {$11$};
			\node at (16.5, 1.1) [inner sep=0.2mm, circle, fill=red!50] {$12$};
			
			\node at (3, -1.3) [] {\large $\TT$};
			\node at (14, -1.3) [] {\large $\PP$};
			\end{tikzpicture}
		\end{center}
		\caption{The edge sets of $\TT$ and $\PP$ consist of all the triangles in the respective figures.}\label{fig:degen-col}
	\end{figure}
	
	We will also use Szemer\'edi's regularity lemma~\citep{szem}; to state the lemma, we need some more notation. Given a graph $G$, and two disjoint nonempty sets of vertices $X, Y \subset V(G)$, we define the density of the pair $(X,Y)$ by
	\[d_G(X,Y) = \frac{|E_G(X,Y)|}{|X||Y|},\]
	and additionally, for $\e >0$, we say that the pair $(X,Y)$ is \emph{$\e$-regular} if we have $|d_G(A,B) - d_G(X,Y)| \le \e$ for all $A \subset X$ and $B\subset Y$ with $|A| \ge \e|X|$ and $|B| \ge \e |B|$. We say that a partition $V_0 \cup V_1 \cup \dots \cup V_k$ of the vertex set of a graph $G$ is an \emph{$\e$-regular partition} if
	\begin{enumerate}
		\item $|V_0| \le \e |V(G)|$,
		\item $|V_1| = |V_2| = \dots = |V_k|$, and
		\item all but at most $\e k^2$ pairs $(V_i, V_j)$ with $1 \le i < j \le k$ are $\e$-regular.
	\end{enumerate}
	In this language, the regularity lemma may be phrased as follows.
	\begin{theorem}\label{sz-reg}
		For every $\e > 0$ and each $t \in \N$, there exists an integer $T$ such that every graph $G$ on at least $T$ vertices admits an $\e$-regular partition $V_0 \cup V_1 \cup \dots \cup V_k$ of its vertex set with $t \le k \le T$.
		\qed
	\end{theorem}
	
	It will also be convenient to have a few well-known consequences of the regularity lemma. The following proposition follows from the fact that the degrees of vertices in regular pairs are typically well-behaved.
	
	\begin{proposition}\label{bigdegsub}
		For every $\alpha > 0$, there exists $\beta > 0$ such that the following holds for all $n\in \N$. In each bipartite graph $G$ between vertex classes $X$ and $Y$ with $\alpha n \le |X|, |Y| \le n$ and $|E(G)| \ge \alpha n^2$, there exists a subset $U \subset X$ with $|U| \ge \beta n$ such that for each $x \in U$, there exists a subset $U_x \subset U$ of size at least $3|U|/4$ with $|N_G(x) \cap N_G(y)| \ge \beta n$ for each $y \in U_x$. \qed
	\end{proposition}
	
	We shall additionally use the following proposition establishing `supersaturation' for $4$-cycles in dense graphs, a special case of the results of Erd\H os and Simonovits~\citep{supersat}.
	\begin{proposition}\label{c4supsat}
		For every $\alpha > 0$, there exists $\beta > 0$ such that every $n$-vertex graph $G$ with $|E(G)| \ge \alpha n^2$ contains at least $\beta n^4$ copies of the $4$-cycle $\CC_4$. \qed
	\end{proposition}
	
	We shall use the multiplicative Chernoff bound; see~\cite{HagRub}, for instance.
	\begin{proposition}\label{chernoff}Let $X$ be a binomial random variable with mean $\mu$. Then for any fixed $\delta>0$, we have
		\[\mathbb{P}(X<(1-\delta)\mu)<\genfrac{(}{)}{}{}{\mathrm e^{-\delta}}{(1-\delta)^{1-\delta}}^\mu.\tag*{\qed}\]
	\end{proposition}
	
	For completeness, we also recall the classification theorem of surfaces; see~\citep{toptext,zip}, for example.
	
	\begin{theorem}\label{classify}
		Every surface is homeomorphic to either the sphere $\ss^2$, a connected sum of finitely many tori, or a connected sum of finitely many real projective planes. \qed
	\end{theorem}
	
	Finally, we need some notation for dealing with hierarchies of constants in our proofs. We shall say that a statement holds `for $\delta \ll \e$' if for any fixed $\e \in (0,1]$, there exists $\delta_\e \in (0,1]$ such that the statement in question holds for all $\delta \in (0, \delta_\e]$. Hierarchies with more constants are defined analogously and are to be read from the right to the left.
	
	To avoid clutter, we shall frequently drop the subscript specifying the graph or $3$-graph in the notation above when the graph or $3$-graph in question is clear, abbreviating, for example, $E_G(X,Y)$ by $E(X, Y)$ or $L_\HH(v)$ by $L(v)$. Additionally, we systematically omit floors and ceilings whenever they are not crucial.
	
	\section{Overview of our strategy}\label{s:outline}
	Let us briefly discuss our approach to establishing Theorem~\ref{mainthm}. To begin with, we discuss finding spanning copies of the sphere since this captures most of the difficulties involved. We shall then say a few words about how we find spanning copies of a general surface.
	
	It is not hard to see that a $3$-graph $\HH$ on $n$ vertices with $\codeg(\HH) \gtrsim 3n/4$ contains a spanning copy of the sphere. This follows from Dirac's theorem; indeed, fix a pair of vertices of $\HH$, say $x$ and $y$, and consider a graph $G$ on $V(\HH) \setminus \{x,y\}$ where two vertices $u$ and $v$ are joined if $x, y \in N_\HH(u,v)$, i.e., if $uvx, uvy \in E(\HH)$. It is easy to verify that $\mindeg (G) \gtrsim n/2$, so it follows from Dirac's theorem that $G$ contains a Hamiltonian cycle. Of course, a Hamiltonian cycle in $G$ translates back to a `spanning double pyramid' in $\HH$ (see Figure~\ref{fig:doubpyr}), and this is of course a spanning copy of the sphere in $\HH$.
	
	\begin{figure}
		\begin{center}
			\begin{tikzpicture}[scale = 0.9]
			
			\node (v1) at (-4, 0) [inner sep=0mm, circle, draw=black!100, minimum size=0.6cm] {$v_1$};
			\node (v2) at (-2, 0) [inner sep=0mm, circle, draw=black!100,minimum size=0.6cm] {$v_2$};
			
			\node (v3) at (2, 0) [inner sep=0mm, circle, draw=black!100, minimum size=0.6cm] {$v_k$};
			\node (v4) at (4, 0) [inner sep=0mm, circle, draw=black!100, minimum size=0.6cm] {$v_1$};
			
			\node (x) at (0, 3) [inner sep=0mm, circle, draw=black!100, minimum size=0.6cm] {$x$};
			\node (y) at (0, -3) [inner sep=0mm, circle, draw=black!100, minimum size=0.6cm] {$y$};
			\draw [thick] (v1) -- (v2) -- (x) -- (v1) -- (y) -- (v2);
			\draw [thick] (v3) -- (v4) -- (x) -- (v3) -- (y) -- (v4);
			\draw [thick, loosely dotted] (v2) -- (v3);
			\end{tikzpicture}
		\end{center}
		\caption{A double pyramid with apexes $x$ and $y$.}\label{fig:doubpyr}
	\end{figure}
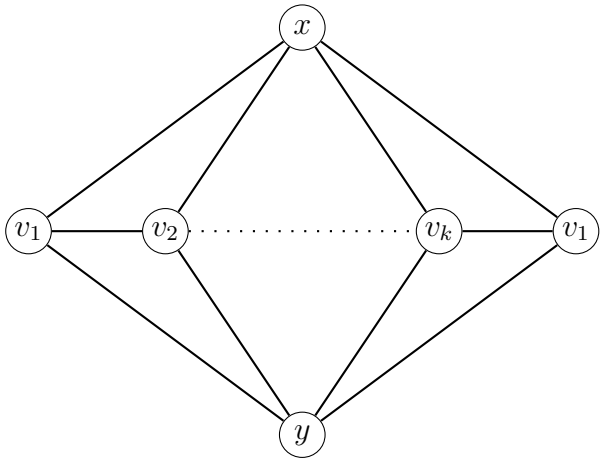
	
	While the above argument is not particularly efficient, it does contain the following useful idea: it is easy to find reasonably large spheres in any dense $3$-graph. Indeed, if we repeat the argument above in a dense $3$-graph $\HH$, then the auxiliary graph $G_{x,y}$ that we construct will be dense for most pairs of vertices $x, y \in V(\HH)$; consequently, $G_{x,y}$ will typically contain long cycles, and these translate back into large double pyramids in $\HH$. This idea may be used to show that a $3$-graph $\HH$ on $n$ vertices with $\codeg(\HH) \gtrsim 2n/3$ contains a spanning copy of the sphere: in such a graph $\HH$, every edge is in a tetrahedron, and we may use this to build `absorbing structures' in the spirit of R\"odl, Ruci\'nski and Szemer\'edi~\citep{rodl-1} that may be used to combine a small number of large spheres that almost span the vertex set into a single sphere spanning the vertex set.
	
	Approaches based on the ideas discussed above however reach a natural barrier at the threshold of $n/2$. Indeed, above this threshold, all the edges of a $3$-graph under consideration necessarily belong to a single tight component, but this is no longer true below this threshold, as shown by the following proposition.
	
	\begin{proposition}\label{toycomp}
		If $\HH$ is a $3$-graph on $n$ vertices with $\codeg(\HH) > (n-3)/2$, then all the edges of $\HH$ belong to a single tight component. Moreover, for each $n \in \N$, there exists a $3$-graph $\HH$ on $n$ vertices with $\codeg(\HH) = \lfloor (n-3)/2 \rfloor$ whose edge set decomposes into two tight components.
	\end{proposition}
	\begin{proof}
		First, let $\HH$ be a $3$-graph on $[n]$ with $\codeg(\HH) > (n-3)/2$. Let us induce an edge-colouring of the complete graph on $[n]$ by setting the colour of an edge $e$ of the complete graph to be the tight component corresponding to all the edges of $\HH$ containing $e$; that this colouring is well-defined follows immediately from the definition of a tight component. Now, observe that no two edges of the complete graph incident to the same vertex can be coloured differently; indeed, if $xy$ and $xz$ constitute such a pair of edges for some $x,y,z\in[n]$, then it is clear that $z \notin N_\HH(x,y)$, $y \notin N_\HH(x,z)$ and $N_\HH(x,y)\cap N_\HH(x,z) = \emptyset$, which leads to an easy contradiction since $\codeg(\HH) > (n-3)/2$. It follows instantly that this induced edge-colouring of the complete graph uses only one colour since the complete graph is connected (and every edge is coloured). Therefore, all the edges of $\HH$ belong to a single tight component, proving the first part of the claim.
		
		Next, given $n \in \N$, let $X$ and $Y$ be two disjoint sets of $\lfloor n/2 \rfloor$ and $\lceil n/2 \rceil$ vertices respectively, and let $\HH$ be the $3$-graph on the vertex set $X \cup Y$ whose edge set consists of all the triples meeting $Y$ in an odd number of vertices. It is easy to see that $\codeg(\HH) = \lfloor (n-3)/2 \rfloor$, and it is not hard to verify that the edge set of $\HH$ decomposes into two tight components, one consisting of all the triples meeting $Y$ in three vertices and the other consisting of all the triples meeting $Y$ in one vertex. This construction proves the second part of the claim.
	\end{proof}
	To get down to the threshold of $n/3$, we need to work somewhat harder. As remarked earlier, it is clear that only those edges from a spanning tight component of a $3$-graph $\HH$ are of any use in constructing a spanning copy of the sphere in $\HH$. An argument similar to the one used to prove Proposition~\ref{toycomp} shows that a $3$-graph $\HH$ on $n$ vertices whose minimum codegree exceeds $n/3$ contains a spanning tight component. However, to prove Theorem~\ref{mainthm}, it will be necessary to say something more about the structure of, and the interaction between, the tight components of a $3$-graph $\HH$ on $n$ vertices with $\codeg(\HH) \gtrsim n/3$.
	
	First, we shall demonstrate that at least one of the spanning tight components of such a $3$-graph has reasonably good `connectibility properties'; this will help us in forming connected sums of smaller surfaces that we will build over the course of our proof. Next, we shall show that it is possible to find a small number of spheres that almost span the vertex set, while crucially ensuring that these spheres meet sufficiently many edges with the good connectibility properties. Finally, we shall use the fact that the $3$-graphs $\TT$ and $\PP$, which respectively represent the torus and the real projective plane, are degenerate to either add handles or crosscaps as needed. In order to implement these ideas, we shall rely on a number of techniques from extremal and probabilistic combinatorics, including those of absorption, regularity and supersaturation.
	
	\section{Proof of the main result}\label{s:proof}
	
	We shall divide the proof of Theorem~\ref{mainthm} into stages, each addressing different aspects of the strategy outlined in the previous section.
	
	Let us point out two notational conventions that we adopt. Given a subgraph $\GG$ of a $3$-graph $\HH$, the subset of the vertex set $V(\HH)$ spanned by $\GG$ is denoted by $V(\GG)$; in the case where $\GG$ consists of a single edge $e \in E(\HH)$, we abuse notation slightly and write $V(e)$ for $V(\GG)$. To avoid clutter, we will also talk about `spheres in a $3$-graph' when, strictly speaking, we mean `copies of the sphere $\ss^2$ in a $3$-graph', and similarly for the torus and the real projective plane.
	
	We shall make use of the following simple averaging lemma at several points.
	
	\begin{lemma}\label{averaging}For any $r\in \mathbb N$ and $\gamma\in(0,1)$, the following holds for all sufficiently large $m\in\mathbb N$. For any finite set $\mathcal X$, and any system $X_1,\dots,X_m$ of subsets of $\mathcal X$ satisfying \[\sum_{i=1}^m|X_i|\ge \gamma m|\mathcal X|,\]
		we may find a set $K\subset[m]$ with $|K|=r$ such that
		\[\left|\bigcap_{i\in K}X_i\right|\ge\gamma^r|\mathcal X|/2.\]
	\end{lemma}
	\begin{proof}Choose a random subset $K\subset[m]$ of size $r$. For each $x\in\mathcal X$, $x$ belongs to $\cap_{i\in K}X_i$ with probability $\binom{f(x)}r\big/\binom mr$, where $f(x)$ is the number of indices $i\in[m]$ such that $x\in X_i$. By Jensen's inequality, we have
		\[\mathbb E\left(\left|\bigcap_{i\in K}X_i\right|\right)\geq|\mathcal X|\binom{\gamma m}r \binom mr^{-1}\]
		provided that $\gamma m>r$. It is now straightforward to prove that the claim holds for all sufficiently large $m \in \N$.\end{proof}
	
	\subsection{Double pyramids in dense 3-graphs}
	Recall that a $3$-graph is said to be \emph{degenerate} or \emph{$3$-partite} if its vertex set may be partitioned into three classes in such a way that each edge meets all three classes. The following lemmas make precise our earlier observation that it is easy to find large spheres in any dense $3$-graph.
	
	\begin{lemma}\label{find-sphere}Let $1/n\ll \delta, 1/k\ll 1$, let $\HH$ be a $3$-graph on $n$ vertices, and let $T\subset V(\HH)$ meet at least $\delta n^3$ edges of $\HH$. Then $\HH$ contains a double pyramid on $2k+2$ vertices with apexes in $T$.\end{lemma}
	\begin{proof}
		For each vertex $v\in T$, recall that  $L(v)$ denotes the graph on $V(\HH) \setminus \{v\}$ with $xy\in E(L(v))$ whenever $vxy\in E(\HH)$. Applying Lemma~\ref{averaging} to these graphs  ---  with $r=2$ and $\mathcal X=V(\HH)^2$  ---  we can find a pair of vertices $v,w \in T$ such that there are at least $\delta^2 n^2/2$ edges common to $L(v)$ and $L(w)$. By Theorem~\ref{kov}  ---  applied with $r=2$  ---  we may find a copy of the cycle $\CC_{2k}$ among these common edges. Such a cycle yields a double pyramid in $\HH$ with apexes $v$ and $w$ and $2k+2$ vertices in total.
	\end{proof}
	
	\begin{lemma}\label{doubpy}
		Let $1/n\ll \delta,\e,\eta \ll 1$, and let $\HH$ be a $3$-partite $3$-graph with vertex classes $A$, $B$ and $T$ with $|A|=|B|=n$ and $|T|\ge \e n$ such that for each $A'\subset A$ and $B'\subset B$ with $|A'|,|B'|\ge \e n$, there are at least $\delta\abs{A'}\abs{B'}\abs{T}$ edges in $\HH[T\cup A'\cup B']$. Then $\HH$ contains at most $\eta n$ disjoint spheres that each have $2$ vertices in $T$ and together cover at least $(1-\e)n$ vertices from each of $A$ and $B$.
	\end{lemma}
	
	\begin{proof} Fix an integer $k$ such that  $1/n\ll 1/k \ll \delta,\e,\eta$. Iteratively, remove spheres $S$ from $\HH$ with $|V(S)\cap A|=|V(S)\cap B|\ge k$ and $|V(S)\cap T|=2$, until it is not possible to find another sphere with these properties. This removes at most $\eta n$ disjoint spheres. Suppose that we are left with $A'\subset A$, $B'\subset B$ and $T'\subset T$.
		
		We will show that $|A'| = |B'| < \e n$, completing the proof. Suppose for the sake of contradiction that $|A'|=|B'|\ge \e n$. Note that we have \[|T'|\ge |T| - 2n/k \ge (1-\delta/2)|T|,\] so there are at least $\delta\abs{A'}\abs{B'}\abs{T}/2$ edges in $\HH[T'\cup A'\cup B']$. Applying Lemma~\ref{find-sphere} to $\HH[T'\cup A'\cup B']$, we may find a double pyramid on $2k+2$ vertices with apexes in $T'$. Since $\HH$ is $3$-partite, the remaining $2k$ vertices are evenly divided between $A'$ and $B'$, yielding a contradiction.
	\end{proof}
	
	\subsection{Colouring and connecting edges}
	Next, we detail how we will connect spheres together into a larger sphere. In this section, we work towards showing that the edges in our $3$-graph can be (mostly) coloured red and green so that any disjoint edges of the same colour can be connected together into a sphere (and, moreover, in many different ways). Thus, two disjoint spheres which share some colour among their edges can be connected together into a larger sphere.
	
	More precisely, given an integer $k\in\N$ and a $3$-graph $\HH$ on $n$ vertices, we say two edges $e,f \in E(\HH)$ are \emph{$(\al,k)$-connectible} if, for some $l$ with $1\leq l\leq k$, there are at least $\al n^l$ sets $A\subset V(\HH)$ with $|A|=l$ for which there is a sphere in $\HH$ containing the edges $e$ and $f$ with vertex set $A\cup V(e)\cup V(f)$. This will be the most convenient definition to use in proving connectibility properties, however, we shall then only use the following simple consequence of $(\al,k)$-connectibility.
	
	\begin{lemma}\label{disjoint}
		If two edges $e,f\in E(\HH)$ are $(\al,k)$-connectible, then for some $l$ with $1\leq l\leq k$ there are at least $\al n/l$ pairwise disjoint sets $A\subset V(\HH)$ with $|A|=l$ for which there is a sphere with vertex set $A\cup V(e)\cup V(f)$ in $\HH$ containing the edges $e$ and $f$.
	\end{lemma}
	\begin{proof}
		Let $l$ with $1\leq l\leq k$ be a number witnessing the fact that $e,f$ are $(\al,k)$-connectible. Choose a maximal collection $A_1,A_2,\dots, A_r$ of disjoint sets with the required properties and note that $U= \cup_{i=1}^r A_i$ comprises $rl$ vertices. By the maximality of our choice, $U$ intersects every set of vertices of size $l$ which can be used to create a sphere including $e$ and $f$. There are at least $\al n^l$ such sets, but at most $rln^{l-1}$ sets of size $l$ meet $U$, so we must have $rl\geq \al n$, as required.
	\end{proof}
	
	Next, we demonstrate that almost all touching pairs of edges in a dense $3$-graph are easily connectible.
	
	\begin{lemma}\label{touchpairs}
		Let $1/n \ll \delta \ll \e < 1$, and let $\HH$ be a $3$-graph on $n$ vertices. Then all but at most $\e n^4$ touching pairs of edges of $\HH$ are contained in at least $\delta n^2$ spheres with $6$ vertices in $\HH$.
	\end{lemma}
	\begin{proof}
		We need to show that if $F$ is any set of $\e n^4$ touching edge-pairs of $\HH$, then at least one of these pairs is contained in at least $\delta n^2$ spheres with $6$ vertices. To this end, fix a set $F$ of $\e n^4$ touching edge-pairs of $\HH$. By averaging, there are two vertices $x,y \in V(\HH)$ such that $e\triangle f = \{x,y\}$ for at least $\e n^2$ edge-pairs $\{e,f\}\in F$, where $\triangle$ denotes the symmetric difference.
		
		Let $G=G_{x,y}$ be the graph on the vertex set $V(\HH)$ where we join two vertices $u$ and $v$ if $uvx,uvy\in E(\HH)$. As $|E(G)|\ge \e n^2$, we know from Proposition~\ref{c4supsat} that $G$ contains $\delta n^4$ copies of the $4$-cycle $\CC_4$. Consequently, there is some $uv\in E(G)$ such that there are at least $\delta n^2$ pairs $w,z \in V(\HH)$ such that the set $\{u, v, w, z\}$ induces a copy of $\CC_4$ in $G$ containing the edge $uv$. For any such a copy of $\CC_4$, note that $\{x,y,u,v,w,z\}$ is the vertex set of a sphere in $\HH$ containing the touching edge-pair $\{uvx,uvy\}$. Thus, we have shown that $F$ contains a pair $\{uvx, uvy\}$ that is contained in at least $\delta n^2$ spheres in $\HH$ with $6$ vertices, proving the claim.
	\end{proof}
	
	Our next (rather coarse) lemma develops the connectibility properties of dense $3$-graphs further.
	\begin{lemma}\label{firstcolour}
		Let $1/n \ll \al \ll \beta \ll \e \ll 1$, and let $\HH$ be a $3$-graph on $n$ vertices with $\codeg(\HH)\ge n/3$. Then all but at most $\e n^3$ edges of $\HH$ can be coloured so that
		\begin{enumerate}
			\item there are at least $\beta n^3/2$ edges of each colour, and
			\item any two disjoint edges of the same colour are $(\al,\kzero)$-connectible.
		\end{enumerate}
	\end{lemma}
	
	\begin{proof} Take $\delta, \eta > 0 $ so that $ \beta \ll\delta\ll\eta\ll\e$ and call a pair of touching edges \textit{good} if it can be completed into a sphere using two more vertices in at least $\delta n^2$ different ways and \textit{bad} otherwise. Note that by Lemma~\ref{touchpairs}, there are at most $\eta n^4$ bad pairs in $\HH$.
		
		We iteratively colour the edges of $\HH$ as follows. Suppose that we have a set $F \subset E(\HH)$ of at least $\e n^3$ uncoloured edges. Let us write $\mathbf{E}$ and $\mathbf{F}$  for the sets of ordered $3$-tuples $(x,y,z)$ such that $xyz \in E(\HH)$ and  $xyz \in F$ respectively. Given a $3$-tuple $\mathbf{f}=(v_1,v_2,v_3) \in \mathbf{F}$, we define
		\[A(\mathbf{f})=\{(x,y,z):v_2v_3x,v_3xy,xyz\in E(\HH)\},\]
		noting that \[|A(\mathbf{f})|\ge (n/3 - 1)(n/3 - 2)(n/3-3) \ge n^3/54.\] Clearly, there are at least $\e n^6/9$ ordered pairs $(\mathbf{f},\mathbf{e})$ of $3$-tuples with $\mathbf{f} \in \mathbf{F}$ and $\mathbf{e}\in A(\mathbf{f})$ since every edge of $F$ corresponds to $6$ different $3$-tuples in $\mathbf{F}$.
		
		For $\mathbf{f} \in \mathbf{F}$, let $B(\mathbf{f})\subset A(\mathbf{f})$ be the set of $3$-tuples $(x,y,z)\in A(\mathbf{f})$ such that the pairs $\{v_1v_2v_3,v_2v_3x\}$, $\{v_2v_3x,v_3xy\}$, and $\{v_3xy,xyz\}$ of touching edges are all good. Note that every bad pair $\{w'x'y',x'y'z'\}$ in $\HH$ arises as one of the three pairs corresponding to $(v_1,v_2,v_3,x,y,z)$ as above for at most $12n^2$ choices of this $6$-tuple. Therefore, there are at most $12\eta n^6$ ordered pairs $(\mathbf{f},\mathbf{e})$ with $\mathbf{f}\in \mathbf{F}$ and $\mathbf{e}\in A(\mathbf{f})\setminus B(\mathbf{f})$.
		
		Now, if $|B(\mathbf{f})|<\eta n^3$ for some $\mathbf{f} \in \mathbf{F}$, then $|A(\mathbf{f})\setminus B(\mathbf{f})| \ge (1/54-\eta)n^3$, so our previous observation implies that there are at most $12\eta n^3/(1/54-\eta)\le 3\e n^3$ different $3$-tuples $\mathbf{f} \in \mathbf{F}$ for which $|B(\mathbf{f})|<\eta n^3$. Hence, there are at least $6\e n^3 - 3\e n^3 \ge 3 \e n^3$ different $3$-tuples $\mathbf{f} \in \mathbf{F}$ for which $|B(\mathbf{f})| \ge \eta n^3$, and correspondingly, there is a set $F'\subset F$ of at least $\e n^3/2$ edges for which at least one of the six $3$-tuples corresponding to that edge has this property.
		
		Next, we create an auxiliary bipartite graph $G$ between $F'$ and $E(\HH)$ with an edge between $f \in F'$ and $e \in E(\HH)$ if some pair $\{\mathbf{f},\mathbf{e}\}$ of $3$-tuples corresponding to these edges of $\HH$ satisfies $\mathbf{e}\in B(\mathbf{f})$; note that $G$ has at least $\e \eta n^6/12$ edges. We appeal to Proposition~\ref{bigdegsub} and find a subset $F''\subset F'$ of size at least $\beta n^3$, such that for any two edges $f_1,f_2\in F''$ there are at least $\beta n^3/2$ edges $f_3\in F''$ with the property that the pairs $f_1,f_3$ and $f_2,f_3$ each have at least $\beta n^3$ common neighbours in $G$.
		
		Thus, for disjoint edges $f_1,f_2\in F''$, there are at least $\beta n^3/4$ edges $f_3\in F''$ disjoint from $f_1$ and $f_2$ so that the pairs $f_1,f_3$ and $f_2,f_3$ each have at least $\beta n^3$ common neighbours in $G$.
		For each such $f_3$, note that there are at least $(\beta n^3)^2/4$ choices of $e_1,e_2\in E(\HH)$ which are disjoint from each other and from $f_1,f_2$ and $f_3$ and such that $f_1e_1,e_1f_3,f_3e_2,e_2f_2\in E(G)$. For such a pair $e_1$ and $e_2$, there are at least $(\delta n^2)^3 - O(n^5)\geq \delta^3 n^6/2$ ways to find a sphere $S_1$ joining $f_1$ and $e_1$ using $6$ extra vertices and avoiding the vertices in $f_3,e_2$ and $f_2$.
		A similar argument for each of the edges $e_1f_3, f_3e_2, e_2f_2 \in E(G)$ in turn shows that we can find spheres $S_2$, $S_3$ and $S_4$ containing the pairs $e_1f_3$, $f_3e_2$ and $e_2f_2$ in turn, each with $6$ vertices in addition to those in the edge pair, and so that all the additional vertices are distinct and not in any of the edges $f_1,f_2,f_3,e_1$ and $e_2$. Furthermore, when choosing each sphere we have at least $\delta^3 n^6/2$ choices. Thus, we can choose such spheres in at least $\delta^{12}n^{24}/16$ different ways. Note that the union of $S_1,S_2,S_3$ and $S_4$ is a sphere containing $f_1$ and $f_2$ and $\kzero$ extra vertices.
		
		There were at least $\beta n^3/4$ choices of $f_3$ and $(\beta n^3)^2/4$ choices of $e_1$ and $e_2$, so thus in total we can find at least $\delta^{12}\beta^3n^{33}/256\geq \alpha n^{33}$ spheres containing $f_1$, $f_2$ and $33$ extra vertices. Thus, $f_1$ and $f_2$ are $(\al,\kzero)$-connectible. As $f_1$ and $f_2$ were arbitrary disjoint edges in $F''$, we can  colour all the edges in $F''$ using a new colour.
				
		It is clear that this procedure colours the edges of $\HH$ as required, thereby proving the lemma.
	\end{proof}
	
	The next observation refines the previous lemma, and is the starting point of our characterisation of the connectibility properties of the tight components of a sufficiently dense $3$-graph.
	
	\begin{lemma}\label{secondcolour}
		Let $1/n \ll \al, 1/k\ll \e \ll 1$, and let $\HH$ be a $3$-graph on $n$ vertices with $\codeg(\HH)\ge n/3$. Then all but at most $\e n^3$ edges of $\HH$ can be coloured so that any two disjoint edges of the same colour are $(\al,k)$-connectible and there are at most $\e n^4$ pairs of touching edges with different colours.
	\end{lemma}
	\begin{proof} Take $\delta, \eta > 0 $ so that $\al,1/k\ll\delta \ll \eta \ll\e$ and set $k_0 = \kzero$. By Lemma~\ref{touchpairs}, there are at most $\eta^2\e n^4/2$ touching edge-pairs in $\HH$ which are not contained in at least $\delta n^2$ spheres with size 6 in $\HH$. By Lemma~\ref{firstcolour}, we can colour the edges of $\HH$ so that at most $\e n^3$ edges are uncoloured, there are at least $\eta n^3$ edges of each colour, and any two disjoint edges of the same colour are $(\delta,k_0)$-connectible. Let $l$ be the number of colours used by this colouring $C_0$, and note that
		\[l\le\binom n3\left(\eta n^3\right)^{-1}<1/\eta.\]
		
		We iteratively construct a sequence of colourings $(C_i)_{1\leq i \leq m}$ as follows. For each $i\ge 0$, if there are at least $\eta^2\e n^4$ touching pairs of edges each using two fixed colours $c_1,c_2$ of $C_i$, then we define the colouring $C_{i+1}$ by replacing  $c_1,c_2$ with a common new colour (if there is more than one choice for such a pair of colours, we pick any acceptable pair). Of course, these recolourings may be performed at most $l-1$ times, and we stop after $m<l$ steps when there are at most \[\binom{l}2 \eta^2 \e n^4<\e n^4\] non-monochromatic touching edge-pairs in the colouring $C_m$.
		
		For $0 \le i \le m$, we prove by induction that any two disjoint edges of the same colour in $C_i$ are $(\delta^{3^i},3^ik_0)$-connectible in $\HH$. This is certainly true for $i=0$, so suppose that we have established this fact for $i$, and that $e,f \in E(\HH)$ have the same colour in $C_{i+1}$ but not in $C_{i}$, where they are coloured, say, red and blue respectively. Then there are at least $\eta^2\e n^4$ touching red-blue edge pairs in $C_{i}$, and at least $\eta^2 \e n^4/2$ of these pairs are contained in at least $\delta n^2$ spheres with size 6 in $\HH$. At least half of the latter pairs $(e',f')$  satisfy the additional requirement that $V(e) \cup V(f)$ and $V(e') \cup V(f')$ are disjoint, since there are only $O(n^3)$ 
		touching-edge pairs which fail to satisfy it.
		By the induction hypothesis and the  pigeonhole principle, we observe that for some $k_1,k_2\le 3^{i-1}k_0$, there are at least
		\[\eta^2 \e n^4/4(3^{i-1}k_0)^2 > \delta n^4\]
		of the above red-blue pairs of touching edges $(e',f')$ such that there are
		\begin{enumerate}
			\item\label{complete1} at least $\delta^{3^{i-1}}n^{k_1}$ ways to complete $e$ and $e'$ into a sphere $S_1$ using $k_1$ extra vertices,
			\item\label{complete2} at least $\delta^{3^{i-1}}n^{k_2}$ ways to complete $f'$ and $f$ into a sphere $S_2$ using $k_2$ extra vertices, and
			\item\label{complete3} at least $\delta n^2$ ways to complete $e'$ and $f'$ into a sphere  $S_3$ using two extra vertices.
		\end{enumerate}
		Thus, there are at least $\delta^{2(1 + 3^{i-1})}n^{k_1 + k_2 + 6}< 2 \delta^{3^{i}} n^{k_1 + k_2 + 6}$ ways to choose a combination of these objects, namely a pair $(e',f')$ and three spheres $S_1$, $S_2$ and $S_3$ with the above properties. Note that for every such choice, the symmetric difference of the edge sets of $S_1$, $S_2$ and $S_3$ forms a sphere containing $e$ and $f$ and $k_1+k_2+6\leq 2\cdot 3^{i-1}k_0 +6 < 3^i k_0$ additional vertices, provided the sets of extra vertices used in \ref{complete1}, \ref{complete2} and \ref{complete3} above are disjoint from each other and from $V(e)\cup V(e')\cup V(f') \cup V(f)$. As the number of ways to choose any triple of vertex sets, let alone spheres, with $k_1, k_2$ and $2$ elements respectively, such that two of the sets intersect or one intersects a given $10$-element set, is $O(n^{k_1 + k_2 + 1})$, we deduce that there are at least
		\[2\delta^{3^{i}} n^{k_1 + k_2 + 6} - O(n^{k_1 + k_2 + 5})\]
		spheres in $\HH$ including $e$, $f$ and $k_1 + k_2 + 6$ extra vertices,  so it follows that $e$ and $f$ are $(\delta^{3^i},3^ik_0)$-connectible in $\HH$.
		
		It is now clear that our final colouring $C_m$ has the required properties, completing the proof.
	\end{proof}
	
	\subsection{Colour interaction}
	We say that a $3$-graph $\HH$ on $n$ vertices is \emph{$(\e,\mu)$-coloured} if $\codeg(\HH) \ge (1/3+\mu)n$ and all but $\e n^3$ of its edges are coloured red or green so that
	\begin{enumerate}
		\item there are at most $\e n^4$ pairs of touching red and green edges, and
		\item at least $\mu n/4$ vertices are contained in fewer than $\e n^2$ red edges.
	\end{enumerate}
	
	The following lemma turns the $(\e,\mu)$-coloured property into a property
	of the link graphs $L(v)$ we will use in Section~\ref{almost} to construct almost spanning spheres.
	For an edge-coloured $3$-graph $\HH$ as above, and $v \in V(\HH)$, we define the {\em green link graph} $GL(v)$ to be the graph on $V(\HH) \setminus \{ v \}$ where two vertices $x$ and $y$ are joined with an edge if $vxy$ is a \emph{green edge} of $\HH$.
	\begin{lemma}\label{newlink} Let $1/n \ll \e \ll  \mu\ll 1$, and let $\HH$ be an $n$-vertex $(\e,\mu)$-coloured $3$-graph. Then, for all but at most $\mu^{4} n$ vertices $v \in V(\HH)$, the green link graph $GL(v)$ has at least $(1/3+\mu/2)n$ vertices with degree at least $(1/3+\mu /2)n$.
	\end{lemma}
	\begin{proof}
		We construct an auxiliary red-green coloured graph $G$ on the vertex set $V(\HH)$ as follows. For each pair $x,y\in V(\HH)$ such that there are at least $(1/3+3\mu/4)n$ edges of one colour in $\HH$ containing $x$ and $y$, add the edge $xy$ to $G$, giving it that colour, where if this pair satisfies this condition for both colours, we colour $xy$ green. It suffices to show that all but at most $\mu^{4} n$ vertices are in at least $(1/3+\mu/2)n$ green edges of $G$, since if $xy$ is a green edge of $G$ then there are at least $(1/3+3\mu/4)$ green edges of the form $xyz$, and it follows that $y$ has degree at least $(1/3+3\mu/4)$ in $GL(x)$.
		
		If $xy$ is not an edge of $G$, then the edges of $\HH$ containing $x$ and $y$ either include at least $\mu n/8$ uncoloured edges or include at least $\mu n/8$ edges of each colour. In the latter case, there are at least $\mu^2 n^2/64$ pairs of differently-coloured edges in $\HH$ which touch along $xy$; since there are at most $\e n^4$ pairs of differently-coloured touching edges in $\HH$, no more than $64\e n^2/\mu^2$ non-edges $xy$ can occur for this reason. Similarly, since there are at most $\e n^3$ uncoloured edges in $\HH$, at most $24\e n^2/\mu$ non-edges $xy$ can occur in $G$ for the former reason. Hence there are at most $88\e n^2/\mu^2<\mu^{6}n^2$ non-edges in $G$.
		
		Since $\HH$ is $(\e,\mu)$-coloured, there are at least $\mu n/4$ vertices in fewer than $\e n^2$ red edges of $\HH$; let $A$ be a set of $\mu n/4$ such vertices. Since each red edge of $G$ extends to more than $n/3$ red edges of $\HH$, each vertex in $A$ is in fewer than $6\e n$ red edges of $G$, so there are fewer than $3\e n^2/2<\mu^{6}n^2$ red edges of $G$ meeting $A$.
		
		Let $B$ be the set of vertices with fewer than $\mu n/8$ green edges in $G$ to $A$. Since for each such vertex there are at least $\mu n/8$ non-edges or red edges with the other end in $A$, and there are fewer than $2\mu^{6}n^2$ pairs which are non-edges or red edges meeting $A$, each of which is counted at most twice, we deduce that $|B|<32\mu^{5}n$. Let $B'$ be the set of vertices with degree less than $(1-\mu/8)n$ in $G$. Since $G$ has fewer than $\mu^{6}n^2$ non-edges, we have $|B'|<8\mu^{5}n$.
		
		Now, fix a vertex $v\in V(\HH)\setminus(B\cup B')$, and suppose that $v$ is in fewer than $(1/3+\mu/2)n$ green edges in $G$. Since $v\not\in B$, there are at least $\mu n/8$ choices of $u\in A$ for which $uv$ is a green edge of $G$; by definition, for each choice of $u$ there are at least $(1/3+3\mu/4)n$ choices of $w$ for which $uvw$ is a green edge of $\HH$. Since $v$ is in fewer than $(1/3+\mu/2)n$ green edges in $G$, and $v\not\in B'$, it must be the case that $vw$ is a red edge of $G$ for at least $\mu n/8$ of these choices of $w$, so there are at least $\mu^2 n^2/64$ pairs $(e,f)$ where $e$ is a red edge of $G$, $f$ is a green edge of $\HH$ and $v\in e\subset f$.
		
		However, there are at most $3\e n^3<\mu^{7}n^3$ pairs $(e,f)$ for which $e$ is a red edge of $G$, $f$ is a green edge of $\HH$ and $e\subset f$, since each such pair corresponds to at least $n/3$ differently-coloured touching edge pairs in $\HH$, of which there are at most $\e n^4$. Each such pair corresponds to at most two choices of $v$, so there are at most $128\mu^{5}n$ vertices in $V(\HH)\setminus(B\cup B')$ which are in fewer than $(1/3+\mu/2)n$ green edges of $G$; it follows that, in total, at most $168\mu^5n<\mu^{4}n$ vertices are in fewer than $(1/3+\mu/2)n$ green edges of $G$, as required.\end{proof}
	
	Crucial for our purposes is the fact that any sufficiently dense $3$-graph admits an $(\e,\mu)$-colouring with strong connectibility properties within each colour class.
	
	\begin{lemma} \label{colourthm} Let $1/n\ll \al,1/k \ll \e \ll \mu$, and let $\HH$ be a $3$-graph on $n$ vertices with $\codeg(\HH)\ge (1/3+\mu)n$. Then it is possible to $(\e,\mu)$-colour $\HH$ so that any monochromatic pair of disjoint edges of $\HH$ is $(\al,k)$-connectible.
	\end{lemma}
	\begin{proof}
		Choose $\eta > 0$ so that $\al,1/k \ll\eta\ll\e$. By Lemma~\ref{secondcolour}, all but at most $\eta n^3$ edges of $\HH$ can be coloured so that any two disjoint edges of the same colour are $(\al,k)$-connectible and there are at most $\eta n^4$ pairs of touching edges with different colours. It only remains to show then that some two colours account for all but $(\e-\eta)n^3$ of the coloured edges, and that at least $\mu n/4$ vertices appear in fewer than $\e n^2$ edges of one of these colours.
		
		To prove this, we shall first show that for all but a few pairs of vertices, there are a large number of edges of a single colour containing that pair which also meet a large number of edges of the same colour along each of the other pairs they contain. In order to consider pairs of vertices in this way, we will find it convenient to define an auxiliary coloured graph $G$. We then show that there is only room in $\HH$ for such large monochromatic structures in at most two colours; this is similar to proving that a $3$-graph with minimum codegree exceeding $n/3$ can have at most two tight components. Finally, we shall use a double counting argument on the interactions between colours to show that there is a similar obstacle to nearly all vertices meeting many edges of both colours.
		
		We start by defining an auxiliary graph $G$ on $V(\HH)$ where we join a pair $x,y \in V(\HH)$ of vertices if there is a unique colour $c$ so that there are at least $(1/3+\mu-\e)n$ edges with colour $c$ in $\HH$ containing $x$ and $y$; moreover, we assign the colour $c$ to such an edge. If $xy$ is not an edge of $G$, then it follows that at least \[(\codeg(\HH)-\e n)\e n>\e n^2/3\] pairs of differently-coloured touching edges of $\HH$ have intersection $\{x,y\}$. Thus, $G$ is `nearly complete', missing at most $3\eta n^2/\e$ edges.
		\begin{claim}
			At most $\e^2n^2$ pairs of vertices are in fewer than $(1/3+\mu-\e)n$ monochromatic triangles of $G$.
		\end{claim}
		\begin{poc}
			Since $G$ has at most $3\eta n^2/\e< \e^3n^2/6$ non-edges, fewer than $\e^3n^3/6$ edges of $\HH$ contain a non-edge of $G$. For every coloured edge of $\HH$ which contains an edge of $G$ coloured differently, there are at least $(1/3+\mu-\e)n$ pairs of differently-coloured touching edges of $\HH$, so we conclude that at most $\e^3n^3/6$ such edges exist. Therefore, all but at most $\e^3n^3/3$ edges of $\HH$ correspond to monochromatic triangles in $G$. Consequently, at most $\e^2n^2$ pairs of vertices lie in more than $\e n$ edges of $\HH$ which do not correspond to a monochromatic triangle, so all other pairs of vertices lie in at least $(1/3+\mu-\e)n$ monochromatic triangles in $G$, proving our claim.
		\end{poc}
		
		Now, let us remove all the edges of $G$ which are in fewer than $(1/3+\mu-\e)n$ monochromatic triangles of $G$, resulting in a graph missing at most \[3\eta n^2/\e + \e^2 n^2 < 2\e^2 n^2\] edges, and then remove, iteratively, at most $\e n/2$ vertices with minimum degree from this graph to obtain a graph $G'$ of minimum degree at least $(1-4\e)n$. In $G'$, every edge is in at least $(1/3+\mu/2)n$ monochromatic triangles, since any edge of $G'$ was in at least $(1/3+\mu-\e)n$ monochromatic triangles of $G$, and in passing to $G'$, edges have been removed from at most $8\e n$ such triangles and vertices have been removed from at most $\e n/2$ of these triangles. Calling a triangle {\em rainbow} if its edges have three different colours, we then have the following claim.
		
		\begin{claim} \label{rainbow}
			There are at most $\e^3 n^3$ rainbow triangles in $G'$.
		\end{claim}
		\begin{poc}
			If $xyz$ is a rainbow triangle in $G'$, say with $xy$ red, $yz$ green and $xz$ blue, then write $W_r$ for the set of vertices $w\in V(\HH)$ such that $wxy$ is a red edge of $\HH$, and define $W_g$ and $W_b$ analogously. We clearly have \[|W_r|,|W_g|,|W_b|\ge(1/3+\mu-\e)n,\] so we can find $\mu n$ vertices in more than one set, and hence $\mu n$ pairs of differently-coloured touching edges. Each such pair arises from at most two rainbow triangles, so it follows that there are at most $2\eta n^3) \mu < \e^3 n^3$ rainbow triangles in $G'$.
		\end{poc}
		
		Next, suppose there are at least three colours in $G'$, say red, green and blue. Each colour meets at least $(1/3+\mu/2)n$ vertices by the definition of $G'$, and at most $\e n$ vertices are in $\e^2 n^2$ rainbow triangles by Claim~\ref{rainbow}, so we may pick a vertex $v$ incident to edges of at least two colours, red and green say, which is in fewer than $\e^2 n^2$ rainbow triangles. Let $V_r$ and $V_g$ be the red and green neighbourhoods of $v$ in $G'$. Since $|V_r\cup V_g|\ge(2/3+\mu)n$, at least $3\mu n/2 > 2\e n$ vertices in $V_r\cup V_g$ meet blue edges, and at least one of these vertices is in fewer than $\e n$ rainbow triangles containing $v$. Pick such a vertex $w$, and assume without loss of generality that $w\in V_r$. Now $w$ has at least $(1/3+\mu/2)n$ blue neighbours in $V(G') \setminus V_g$ and at least $(1/3+\mu/2)n$ red neighbours in $V_r$, so $|V(G') \setminus V_g|\ge(2/3+\mu)n$, which is a contradiction.
		
		Therefore, $G'$ only has two colours, say, red and green. Write $B_r$ and $B_g$ for the set of vertices spanned by the red and green edges respectively, and suppose without loss of generality that $|B_r|\le|B_g|$.
		\begin{claim} \label{Br}
			$|B_r|\le (1-\mu/4-6\e)n$, whence $|V(G')\setminus B_r|\ge(\mu/4+3\e)n$.
		\end{claim}
		\begin{poc}
			Suppose, for the sake of a contradicton, that \[|B_g|\ge|B_r|\ge (1-\mu /4-6\e)n.\] At least $(1-\mu/2-12\e)n$ vertices are in both $B_r$ and $B_g$, and each such vertex is incident to at least $(1/3+\mu/2)n$ edges of each colour, and to at least $(1-4\e)n$ edges in total by the definition of $G'$. Consequently, for each such vertex $v$, there are at least \[(1/3+\mu/2)(2/3-\mu/2-4\e)n^2\geq n^2/9\] pairs $x,y \in V(G')$ with $vx$ red and $vy$ green. Each such pair produces a triple which is not a monochromatic triangle in $G'$, and no such triple is counted more than twice, so there are at least $n^3/9$ distinct triples of this kind. However, $G'$ has at least $(1-\e/2)(1-4\e)n^2/2$ edges and each edge of $G'$ is in at least $(1/3+\mu/2)n$ monochromatic triangles, so the number of triples inducing monochromatic triangles is at least \[(1-\e/2)(1-4\e)(1/3 + \mu/2)n^3/3 > n^3/18.\] We therefore find at least
			\[n^3/9+n^3/18>\binom n3\] distinct triples on $V(G')$, which is a contradiction.
		\end{poc}

		We now obtain the colouring we seek as follows: we un-colour all the edges of $\HH$ not coloured red or green, as well as all the edges of $\HH$ containing either a non-edge of $G'$ or a vertex not in $V(G')$. Let us count the number of edges of $\HH$ that are no longer coloured after we do this. At most $\eta n^3$ edges of $\HH$ were not coloured in the original colouring, at most $2\e^2n^3$ edges of $\HH$ contain a non-edge of $G'$, and at most $\e n^3/2$ edges of $\HH$ contain a vertex not in $V(G')$. Finally, any edge of $\HH$ containing an edge of $G'$ not coloured either red or green contributes at least $n/3$ pairs of differently-coloured touching edges, so there are at most $3\eta n^3$ such edges. Therefore, the number of edges of $\HH$ not coloured in our final colouring is at most
		\[\eta n^3 + 2\e^2n^3 + \e n^3/2 + 3\eta n^3 < \e n^3.\]
		Finally, notice that any red edge in our final colouring of $\HH$ incident to a vertex in $V(G') \setminus B_r$ is in at least $n/3$ pairs of differently-coloured touching edges, so the number of such edges is at most $3\eta n^3 < \e^2n^3$. It follows from Claim~\ref{Br} that at least $\mu n/4$ vertices in $V(G') \setminus B_r$ meet fewer than $\e n^2$ red edges, as required. All the other properties that we need are inherited from the original colouring of $\HH$, proving the statement.
	\end{proof}

	\subsection{Green-tinged absorbers}
	We now build absorbing structures in an $(\e, \mu)$-coloured $3$-graph. These structures will be crucial in transforming an `almost spanning' copy of a surface into a spanning one.
	
	Fix an $(\e, \mu)$-coloured $3$-graph $\HH$. We say a subgraph of $\HH$ is \emph{green-tinged} if it has at least two green edges. Given a subset $U$ of the vertices of $\HH$, a \emph{green-tinged absorber for $U$} is a sphere $S$ in $\HH[V(\HH) \setminus U]$ which contains two green edges $e$ and $f$ so that for every $U'\subset U$, there is a sphere in $\HH$ on the vertex set $V(S)\cup U'$ containing both $e$ and $f$.
	
	\begin{lemma}\label{greenabsorb}
		Let $1/n \ll \e \ll \mu \ll 1$, $1/n\ll\eta$, and suppose that $\HH$ is an $(\e,\mu)$-coloured $3$-graph with $n$ vertices.
		Then for any $R\subset V(\HH)$ with $|R|\le \mu n/72$, there is a collection of $l\le \eta n$ vertex-disjoint spheres in $\HH[V(\HH) \setminus R]$,
		spanning at most $8|R|$ vertices in total, which are respectively green-tinged absorbers for some pairwise disjoint sets $R_1, R_2, \dots,R_l\subset R$ with $R_1 \cup R_2 \cup \dots \cup R_l = R$.
	\end{lemma}
	\begin{proof}
		Fix an integer $k \ge \mu/(36\eta)$. We shall use a greedy procedure to iteratively find `large' absorbers for subsets of $R$ of size $k$ until fewer than $\eta n/2$ vertices of  $R$ remain, and then find `small' absorbers for the remaining vertices one by one. We shall also ensure that the absorber that we construct for a subset $X \subset R$ at any particular iteration of our procedure uses at most $8|X|$ vertices. Observe that our choice of $k$ ensures that at most $\eta n/2$ large absorbers and $\eta n/2$ small absorbers are built.
		
		Let us now describe a step of our iterative procedure. Write $A$ for the union of the vertices in the vertex-disjoint absorbers selected so far, and $R'$ for the remaining vertices in $R$ for which we have yet to built an absorber; in particular, we initially have $A=\emptyset$ and $R'=R$. Note that we shall always ensure that $|A\cup R|<9|R|\le \mu n/8$.
		
		We first describe how we proceed when $|R'|\ge \eta n/2$. Let $B\subset V(\HH)$ be the set of vertices contained in at most $\e n^2$ red edges, so that $|B|\ge \mu n/4$, and choose a set $B'\subset B\setminus (A\cup R)$ with $|B'|=\mu n/8$. For each $v\in R'$, let $L_v$ be the subgraph of the link-graph $L(v)$ induced by the set $V(\HH)\setminus (A\cup R)$. Let $L'_v\subset L_v$ be the subgraph of $L_v$ consisting of those edges of $L_v$ which are additionally in at least $n/3$ green edges of $\HH$. We shall show that $L'_v$ contains many edges incident with $B'$.
		
		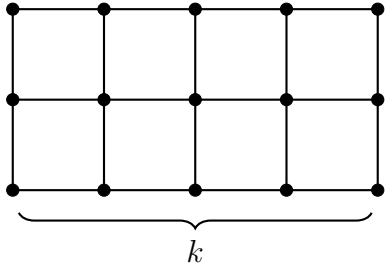
\begin{figure}\begin{center}
				\begin{tikzpicture}[scale = 0.8]
				\draw[thick] (0,0) -- (6,0);
				\draw[thick] (0,1.5) -- (6,1.5);
				\draw[thick] (0,3) -- (6,3);
				\draw[thick] (0,0) -- (0,3);
				\draw[thick] (1.5,0) -- (1.5,3);
				\draw[thick] (3,0) -- (3,3);
				\draw[thick] (4.5,0) -- (4.5,3);
				\draw[thick] (6,0) -- (6,3);
				
				\draw[fill] (0,0) circle [radius=0.1];
				\draw[fill] (1.5,0) circle [radius=0.1];
				\draw[fill] (3,0) circle [radius=0.1];
				\draw[fill] (4.5,0) circle [radius=0.1];
				\draw[fill] (6,0) circle [radius=0.1];
				
				\draw[fill] (0,1.5) circle [radius=0.1];
				\draw[fill] (1.5,1.5) circle [radius=0.1];
				\draw[fill] (3,1.5) circle [radius=0.1];
				\draw[fill] (4.5,1.5) circle [radius=0.1];
				\draw[fill] (6,1.5) circle [radius=0.1];
				
				\draw[fill] (0,3) circle [radius=0.1];
				\draw[fill] (1.5,3) circle [radius=0.1];
				\draw[fill] (3,3) circle [radius=0.1];
				\draw[fill] (4.5,3) circle [radius=0.1];
				\draw[fill] (6,3) circle [radius=0.1];
				
				\draw [thick,decoration={brace,amplitude=0.5em,mirror,raise=0.3cm},decorate] (0.1,0) -- (5.9,0);
				\node at (3,-1) {$k$};
				\end{tikzpicture}
				\caption{The double ladder $DL_k$.}\label{fig:dlk}
		\end{center}\end{figure}
		
		Since $|A\cup R|\le \mu n/8$, it follows that $L_v$ has minimum degree $(1/3+7\mu/8)n$. Every vertex $x\in B'$ has at least $(1/3+3\mu/4)n$ edges in $L_v$ to $V(\HH)\setminus(A\cup R\cup B')$, and by the definition of $B'$, at most  $4\e n/\mu$ of these edges in $L_v$ are contained in more than $\mu n/4$ red edges of $\HH$, for otherwise we would get $(4\e n/\mu) (\mu n/4)=\e n^2$ red edges incident with $x\in B'$. A similar calculation shows that at most $12\e n^2/\mu$ edges of $L_v$ are in more than $\mu n/4$
		uncoloured edges of $\HH$. Easily, if $e$ is an edge of $L_v$ contained in at most $\mu n/2$ edges of $H$ that are either red or uncoloured,
		then $e$ lies in $L'_v$. From these facts, we conclude that $L'_v$ contains more than one third of all edges between $B'$ and $V(\HH)\setminus(A\cup R\cup B')$ for each $v \in R'$.
		
		For a set $U\subset R'$ of size $k$, write $L_U  = \cap_{v \in U} L'_v$. Now, by Lemma~\ref{averaging}  ---  applied with $\mathcal X$ being the set
		$B' \times (V(\HH)\setminus(A\cup R\cup B'))$ of possible edges of $L_U$, $r=k$ and $\gamma=1/3$  ---  we may find a set $U\subset R'$ of size $k$
		such that $L_U$ contains more than a $3^{-k}/2$ proportion of the pairs in $B' \times (V(\HH)\setminus(A\cup R\cup B'))$, so $|E(L_U)|\geq 3^{-k}\mu n^2/17$.
		Each edge in $L_U$ is contained in at least $n/3$ green edges of $\HH$ by definition, so averaging using Lemma~\ref{averaging} again  ---  now with
		$\mathcal X=E(L_U)$, $m=|V(\HH)\setminus(A\cup R)|$, each set $X_i$ being the subset of $E(L_U)$ which form green edges with a given vertex in
		$V(\HH)\setminus(A\cup R)$, $r=k+1$ and $\gamma=1/4$  ---  we can find a set $T=\{t_1,t_2,\dots,t_{k+1}\}$ of $k+1$ vertices disjoint from $U$,
		and a subgraph $L'_U$ of $L_U$ with at least $12^{-k}\mu n^2/136$ edges with the property that each edge in $L'_U$ and each vertex in $T$ together form a green edge of $\HH$.
		
		We can now find a copy of the double ladder $DL_k$ with $k$ spaces, depicted in Figure~\ref{fig:dlk}, in $L'_U$; indeed, since $DL_k$ is bipartite, Theorem~\ref{kov} guarantees the existence of such a copy. The green edges consisting of vertices of $T$ and edges of the copy of $DL_k$ contain an absorbing sphere $S_U$ for $U$, shown in Figure~\ref{fig:absorption}. 
		Since all the edges of $S_U$ containing $t_{k+1}$ are green, and are present no matter which vertices are absorbed, $S_U$ is a green-tinged absorber; it has $4(k+1)\leq 8k$ vertices, as required.

		\begin{figure}\begin{center}
				\begin{tikzpicture}[scale = 1.12]
				\draw[thick] (0,0) -- (4,0);
				\draw[thick,red] (0,1) -- (4,1);
				\draw[thick] (0,2) -- (4,2);
				\draw[thick] (0,0) -- (0,2);
				\draw[thick] (1,0) -- (1,2);
				\draw[thick] (2,0) -- (2,2);
				\draw[thick] (3,0) -- (3,2);
				\draw[thick] (4,0) -- (4,2);
				
				\draw[thick,red] (0,0) -- (1,2);
				\draw[thick,red] (0,2) -- (1,0);
				\draw[thick,red] (1,0) -- (2,2);
				\draw[thick,red] (1,2) -- (2,0);
				\draw[thick,red] (2,0) -- (3,2);
				\draw[thick,red] (2,2) -- (3,0);
				\draw[thick,red] (3,0) -- (4,2);
				\draw[thick,red] (3,2) -- (4,0);
				
				\draw[thick,green] (2,3) .. controls(-1,3) and (-1,1) .. (0,1);
				\draw[thick,green] (2,3) -- (0,2);
				\draw[thick,green] (2,3) -- (1,2);
				\draw[thick,green] (2,3) -- (2,2);
				\draw[thick,green] (2,3) -- (3,2);
				\draw[thick,green] (2,3) -- (4,2);
				\draw[thick,green] (2,3) .. controls(5,3) and (5,1) .. (4,1);
				
				\draw[thick,green] (2,-1) .. controls(-1,-1) and (-1,1) .. (0,1);
				\draw[thick,green] (2,-1) -- (0,0);
				\draw[thick,green] (2,-1) -- (1,0);
				\draw[thick,green] (2,-1) -- (2,0);
				\draw[thick,green] (2,-1) -- (3,0);
				\draw[thick,green] (2,-1) -- (4,0);
				\draw[thick,green] (2,-1) .. controls(5,-1) and (5,1) .. (4,1);
				
				\draw[fill] (0,0) circle [radius=0.075];
				\draw[fill] (1,0) circle [radius=0.075];
				\draw[fill] (2,0) circle [radius=0.075];
				\draw[fill] (3,0) circle [radius=0.075];
				\draw[fill] (4,0) circle [radius=0.075];
				
				\draw[fill] (0,1) circle [radius=0.075];
				\draw[fill] (1,1) circle [radius=0.075];
				\draw[fill] (2,1) circle [radius=0.075];
				\draw[fill] (3,1) circle [radius=0.075];
				\draw[fill] (4,1) circle [radius=0.075];
				
				\draw[fill] (0,2) circle [radius=0.075];
				\draw[fill] (1,2) circle [radius=0.075];
				\draw[fill] (2,2) circle [radius=0.075];
				\draw[fill] (3,2) circle [radius=0.075];
				\draw[fill] (4,2) circle [radius=0.075];
				
				\draw[fill] (0.5,1) circle [radius=0.075] node[below=0.3em,font=\footnotesize]{$t_1$};
				\draw[fill] (1.5,1) circle [radius=0.075] node[below=0.3em,font=\footnotesize]{$t_2$};
				\draw[fill] (2.5,1) circle [radius=0.075] node[below=0.3em,font=\footnotesize]{$t_3$};
				\draw[fill] (3.5,1) circle [radius=0.075] node[below=0.3em,font=\footnotesize]{$t_4$};
				
				\draw[fill] (2,3) circle [radius=0.075] node[above,font=\footnotesize]{$t_{k+1}$};
				\draw[fill] (2,-1) circle [radius=0.075] node[below,font=\footnotesize]{$t_{k+1}$};
				
				\end{tikzpicture}
				\qquad
				\begin{tikzpicture}[scale = 1.12]
				\draw[thick] (0,0) -- (4,0);
				\draw[thick,red] (0,1) -- (4,1);
				\draw[thick] (0,2) -- (4,2);
				\draw[thick] (0,0) -- (0,2);
				\draw[thick] (1,0) -- (1,2);
				\draw[thick] (2,0) -- (2,2);
				\draw[thick] (3,0) -- (3,2);
				\draw[thick] (4,0) -- (4,2);
				
				\draw[thick,red] (0,0) -- (1,1);
				\draw[thick,red] (0,1) -- (1,0);
				\draw[thick,red] (1,0) -- (2,1);
				\draw[thick,red] (1,1) -- (2,0);
				\draw[thick,red] (0,1) -- (1,2);
				\draw[thick,red] (0,2) -- (1,1);
				\draw[thick,red] (1,1) -- (2,2);
				\draw[thick,red] (1,2) -- (2,1);
				\draw[thick,red] (2,0) -- (3,2);
				\draw[thick,red] (2,2) -- (3,0);
				\draw[thick,red] (3,0) -- (4,2);
				\draw[thick,red] (3,2) -- (4,0);
				
				\draw[thick,green] (2,3) .. controls(-1,3) and (-1,1) .. (0,1);
				\draw[thick,green] (2,3) -- (0,2);
				\draw[thick,green] (2,3) -- (1,2);
				\draw[thick,green] (2,3) -- (2,2);
				\draw[thick,green] (2,3) -- (3,2);
				\draw[thick,green] (2,3) -- (4,2);
				\draw[thick,green] (2,3) .. controls(5,3) and (5,1) .. (4,1);
				
				\draw[thick,green] (2,-1) .. controls(-1,-1) and (-1,1) .. (0,1);
				\draw[thick,green] (2,-1) -- (0,0);
				\draw[thick,green] (2,-1) -- (1,0);
				\draw[thick,green] (2,-1) -- (2,0);
				\draw[thick,green] (2,-1) -- (3,0);
				\draw[thick,green] (2,-1) -- (4,0);
				\draw[thick,green] (2,-1) .. controls(5,-1) and (5,1) .. (4,1);
				
				\draw[fill] (0,0) circle [radius=0.075];
				\draw[fill] (1,0) circle [radius=0.075];
				\draw[fill] (2,0) circle [radius=0.075];
				\draw[fill] (3,0) circle [radius=0.075];
				\draw[fill] (4,0) circle [radius=0.075];
				
				\draw[fill] (0,1) circle [radius=0.075];
				\draw[fill] (1,1) circle [radius=0.075];
				\draw[fill] (2,1) circle [radius=0.075];
				\draw[fill] (3,1) circle [radius=0.075];
				\draw[fill] (4,1) circle [radius=0.075];
				
				\draw[fill] (0,2) circle [radius=0.075];
				\draw[fill] (1,2) circle [radius=0.075];
				\draw[fill] (2,2) circle [radius=0.075];
				\draw[fill] (3,2) circle [radius=0.075];
				\draw[fill] (4,2) circle [radius=0.075];
				
				\draw[fill] (0.5,1.5) circle [radius=0.075] node[above,font=\footnotesize]{$v$};
				\draw[fill] (1.5,1.5) circle [radius=0.075] node[above,font=\footnotesize]{$w$};
				\draw[fill] (0.5,0.5) circle [radius=0.075] node[below,font=\footnotesize]{$t_1$};
				\draw[fill] (1.5,0.5) circle [radius=0.075] node[below,font=\footnotesize]{$t_2$};
				\draw[fill] (2.5,1) circle [radius=0.075] node[below=0.3em,font=\footnotesize]{$t_3$};
				\draw[fill] (3.5,1) circle [radius=0.075] node[below=0.3em,font=\footnotesize]{$t_4$};
				
				\draw[fill] (2,3) circle [radius=0.075] node[above,font=\footnotesize]{$t_{k+1}$};
				\draw[fill] (2,-1) circle [radius=0.075] node[below,font=\footnotesize]{$t_{k+1}$};
				\end{tikzpicture}
				\caption{The absorbing sphere $S_U$ for $U$ is shown on the left; each triangle in the figure corresponds to an edge of $\HH$, which exists because the underlying double ladder lies in $L'_U$. On the right, we show how $S_U$ can absorb two arbitrary vertices $v,w\in U$.}\label{fig:absorption}
		\end{center}\end{figure}
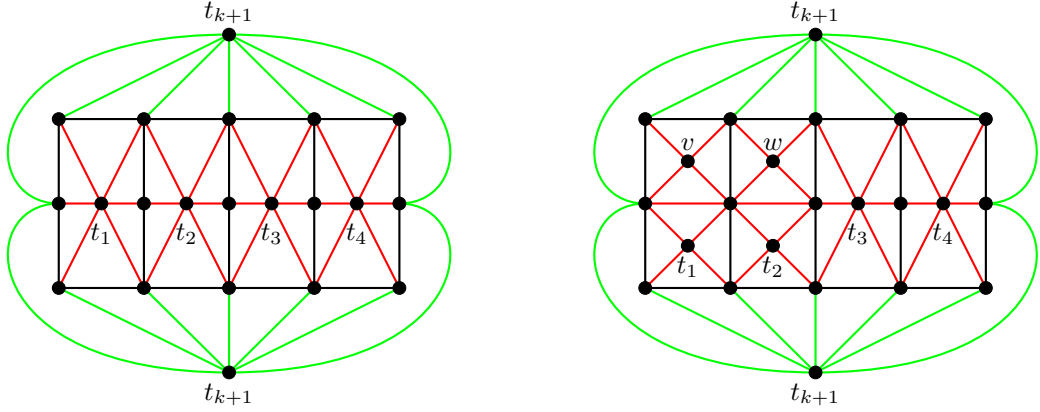
		
		We now turn to dealing with how to proceed when  $|R'|< \eta n/2$. In this case, we find absorbing spheres for the vertices in $R'$ one by one as follows. We pick $v\in R'$, construct the graph $L'_v$ as above, and
		use Lemma~\ref{averaging} with $\mathcal X=E(L'_v)$, $r=2$ and $\gamma=1/4$, to find a pair of vertices $T=\{t_1,t_2\}$ disjoint from $A\cup R$
		and a graph $L''_v\subset L'_v$ with at least \[|E(L'_v)|/32>n^2/800\] edges, each of which forms a green edge of $\HH$ with both $t_1$ and $t_2$. We may find a copy of $DL_1$ in $L''_v$ using Theorem~\ref{kov}, and the corresponding green edges of $\HH$ contain a green-tinged absorber for $\{v\}$ of order $8$, obtained by repeating the construction of Figure~\ref{fig:absorption} with $k=1$.
		
		The iterative procedure described above clearly produces the absorbers we require, proving the lemma.
	\end{proof}
	
	By letting the spheres we construct absorb all of $R$ in Lemma~\ref{greenabsorb}, we immediately obtain the following corollary.
	\begin{corollary}\label{greensphere}
		Let $1/n \ll \e \ll \mu \ll 1$, $1/n\ll \eta$ and suppose that $\HH$ is a $(\e,\mu)$-coloured $3$-graph on $n$ vertices. Then for any $R\subset V(\HH)$ with $|R|\le \mu n/72$, there is a collection of at most $\eta n$ vertex-disjoint green-tinged spheres in $\HH$ which cover $R$. \qed
	\end{corollary}
	
	\subsection{Spheres covering twice as many bad vertices as good vertices}
	The following lemma will allow us to create spheres that use twice as many vertices from some specific set of `bad' vertices as from a specific set of `good' vertices. This will prove useful in Section~\ref{almost} since the connectibility properties of our $3$-graphs will impose constraints of precisely this nature.
	
	\begin{lemma}\label{mopbad} Let $1/N\ll \beta,\gamma,1/k$ and suppose that $\HH$ is a $3$-graph on vertex classes $A$ and $B$, where $A$ has at least $N$ vertices and every edge has one vertex in $B$ and two vertices in $A$. Suppose that for at least $\beta N$ vertices $v\in B$, the link-graph $L(v)$ has at least $(1/2+\gamma)|A|$ vertices with degree at least $(1/2+\gamma)|A|$. Then there is a sphere $S$ in $\HH$ which has $4k$ vertices in $B$ and $2k+2$ vertices in $A$.
	\end{lemma}
	\begin{proof}Write $B'$ for the set of vertices $v\in B$ for which $L(v)$ has at least $(1/2+\gamma)|A|$ vertices with degree at least $(1/2+\gamma)|A|$. For each $v\in B'$, let $H_v$ be the $3$-graph of edges which are triangles in $L(v)$. Choose a vertex $v_1$ with degree at least $(1/2+\gamma)|A|$ in $L(v)$, and note that at least $2\gamma |A|$ of its neighbours must have degree at least $(1/2+\gamma)|A|$ in $L(v)$ also. Thus, choosing $v_2$ to be such a neighbour, there are at least $2\gamma |A|$ triangles containing $v_1$ and $v_2$. As there were at least $2\gamma (1/2+\gamma)|A|^2$ choices for the pair $(v_1,v_2)$, $L(v)$ contains at least $\gamma^2|A|^3/3$ triangles. Hence, $H_v$ has at least $\gamma^2|A|^3/3$ edges.

		By Lemma~\ref{averaging} we can now find a set $K\subset B'$ of size $4k$ for which the $3$-graph $H_K =\cap_{v\in K} H_v$ has at least $(2\gamma^2)^{4k}|A|^3/12$ edges. By Lemma~\ref{find-sphere}, applied with $T=A$ and $\delta=(2\gamma^2)^{4k}/12$, we can find a double pyramid on exactly $2k+2$ vertices in $H_K$. This double pyramid has $4k$ faces, and by putting a vertex from $K$ into each face of this double pyramid, we get a sphere in $\HH$ with $4k$ vertices from $B$ and $2k+2$ vertices from $A$.
	\end{proof}

	\subsection{Almost spanning green spheres}\label{almost}
	
	We now show that we can find a reasonably small set of disjoint spheres which span all but a small fraction of the vertices of an $(\e, \mu)$-coloured $3$-graph $\HH$.
	The proof uses Szemer\'edi's regularity lemma, and the following lemma will be used to partition the reduced graph.
	
	\begin{lemma}\label{matchpart} Let $1/n\ll \e$, and suppose that $G$ is an $n$-vertex graph. Then we can find a partition of $V(G)$ into sets $Z$, $B$, $C$ and $D$
		such that
		\begin{enumerate}
			\item there is a perfect matching in $Z$,
			\item there is a perfect matching between $C$ and $D$, and
			\item $|E(B\cup D,Z\cup B\cup D)|\leq \e n^2$.
		\end{enumerate}
	\end{lemma}
	\begin{proof}
		Let $M$ be a matching of maximum size in $G$, and for a set of vertices $X$, write $N_M(X)$ for the neighbourhood of $X$ in $M$, i.e., for the set of vertices joined to $X$ by edges of $M$. Let $B$ be the set of vertices which are not in $M$. We will find sequences of disjoint sets $C_1,\dots,C_r$ and $D_1,\dots,D_r$ as follows: iteratively, for each $i\geq 1$, if there are at least $\e n/2$ vertices in \[V(M) \setminus \bigcup_{j<i}({C}_i\cup {D}_i)\] with at least $\e n/2$ neighbours in \[{B}\cup\left(\bigcup_{j<i}{D}_j\right),\] then let ${C}_i$ be the set of those vertices and set ${D}_i=N_M({C}_i)$.
		
		Note that the sets $C_1, C_2, \dots, C_r$ as constructed above are disjoint, and each such set has size at least $\e n/2$; this ensures that $r \le 2/\e$. We may also assume that $r\ge 1$, since if not, we are done by setting $Z=V(M)$, $B=V(G)\setminus Z$ and $C=D=\emptyset$. Indeed, then $E(B\cup D,Z\cup B\cup D)=E(B,Z)$ by the maximality of $M$. Since no iterations of our procedure were possible, there are at most $\e n/2$ vertices in $Z$ with more than $\e n/2$ neighbours in ${B}$, contributing at most $\e n^2/2$ edges in total, and the remaining vertices contribute at most $\e n^2/2$ edges, giving the required result.
		
		Now, let ${C}=\cup_{i=1}^r{C}_i$ and let ${D}=\cup_{i=1}^r{D}_i=N_M({C})$, and note that $1 \le r \le 2/\e$. We shall now deduce some properties of the sets of vertices we have constructed thus far.
		
		\begin{claim}\label{claimpath}
			For any vertex $v\in {B}\cup {D}$ and any set $U\subset {B}\cup {C}\cup
			{D}$ with size at most $\e n/8$ such that $v\notin U\cup N_M(U)$, there is, for some $0\leq s\leq 1/\e$,
			a path of length $2s$ from $v$ into ${B}$ in $G[({B}\cup {C}\cup {D})\setminus U]$ which contains
			$s$ edges from $M$.
		\end{claim}
		\begin{poc}
			Starting with $x_0=v$, find the longest sequence $x_0,y_0,x_1,y_1,\dots,y_{s-1},x_s$ of
			distinct vertices in $({B}\cup
			{C}\cup {D})\setminus U$ such that
			\begin{enumerate}[label={(\roman*)}]
				\item for each $0\leq i<s$, $x_iy_i\in M$,
				\item for each $0\leq i\leq s$,  $x_i\in ({B}\cup (\cup_{j\leq
					r-i}{D}_j))\setminus (U\cup N_M(U))$, and
				\item for each $0\leq i<s$, $y_ix_{i+1}\in E(G)$.
			\end{enumerate}
			Note that such a sequence exists as the one-term sequence $x_0$ satisfies these conditions. Note also that if $x_i\in B$ for some $i$, then $i=s$ and the sequence stops at that point since there are no subsequent choices for $y_i$. In particular, since $x_r$, if it exists, must be in $B$, we have $s\le r\le 1/\e$. We claim that $x_s\in {B}$, and that this sequence consequently gives us the desired path.
			
			Suppose for the sake of a contradiction that $x_s\not\in B$. Then\[x_s\in \left(\bigcup_{j\leq r-i}{D}_j\right)\setminus (U\cup
			N_M(U)),\] so choosing $y_s$ such that $x_sy_s\in M$, we have \[y_s\in \left(\bigcup_{j\leq r-i}{C}_j\right)\setminus U.\]
			Note that since $x_iy_i\in M$ for each $i<s$, we have that $y_s$ is distinct from $y_1,\dots,y_{s-1}$.
			
			We know that $y_s\in \cup_{j\leq r-i}{C}_j$, so, by definition, $y_s$ has at least $\e
			n/2$ neighbours in ${B}\cup (\cup_{j\leq r-i}{D}_j)$. Therefore, since
			\[\left|U\cup N_M(U)\cup \{x_0,y_0,x_1,y_1,\dots,y_{s-1},x_s\}\right| \le \e n/4+2/\e<\e n/2,\]
			we can find a vertex $x_{s+1}$ which is a neighbour of $y_s$ satisfying
			\[x_{s+1}\in {B}\cup \left(\bigcup_{j\leq r-i}{D}_j\right)\setminus (U\cup N_M(U)\cup \{x_0,y_0,x_1,y_1,\dots,y_{s-1},x_s,\}),\]
			contradicting maximality of the sequence, and proving the claim.
		\end{poc}
		
		\begin{claim}\label{claimonevx}
			At most one vertex in each edge of $M$ has at least three neighbours in ${B}\cup {D}$.
		\end{claim}
		\begin{poc}
			Suppose to the contrary that for some edge $u_0v_0\in M$, each of $u_0$ and $v_0$
			has at least three neighbours in $B\cup D$.
			Then there exist
			distinct vertices $u_1$ and $v_1$ in $({B}\cup {D})\setminus\{u_0,v_0\}$
			with $u_0u_1, v_0v_1\in E(G)$.
			Note that $B\cup D$ contains no edges of $M$ by definition,
			so we necessarily have $u_1\neq N_M(v_1)$.
			
			Appealing to Claim~\ref{claimpath}, we now find an integer $0\leq s_1\leq r$ and a path
			$P_1$ in $G[({B}\cup {C}\cup {D})\setminus\{u_0,v_0,v_1,N_M(v_1)\}]$ with length
			$2s_1$ from $u_1$ into ${B}$ which contains $s_1$ edges in $M$. Note that $N_M(V(P_1))\subset V(P_1)$, since this path contains edges of $M$ covering all but the last vertex of the path, which is in $B$ and hence not in $V(M)$.
			
			Using Claim~\ref{claimpath} again, we find another integer $0\leq s_2\leq r$ and a path
			$P_2$ in $G[({B}\cup {C}\cup {D})\setminus(\{u_0,v_0\}\cup V(P_1))]$ with length
			$2s_2$ from $v_1$ into ${B}$ which contains $s_2$ edges in $M$.
			
			Switching edges into the matching $M$ along $P_1$ and $P_2$ creates a
			matching in $G$ with the same number of edges as $M$ which does not
			contain either $u_1$ or $v_1$. Removing $u_0v_0$ from $M$ and adding $u_0u_1$
			and $v_0v_1$ thus creates a matching larger than $M$ in $G$, which is a contradiction.
		\end{poc}
		
		By construction, every vertex in $C$ has at least $\e n/2$ neighbours in $B\cup D$, so Claim~\ref{claimonevx} implies that $C$ contains no edges of $M$. Since $C\subset V(M)$,
		and ${D}=N_M({C})$, we have that $C\cap D=\varnothing$ and that $M$ contains a perfect matching between $C$ and $D$. Now, let $Z=V(M)\setminus (C\cup D)$ and note that $M$ also contains a perfect matching on $Z$, and that $Z$, $B$, $C$ and $D$ partition $V(G)$. It therefore suffices to show that $|E(B\cup D,Z\cup B\cup D)|\le \e n^2$. To do so, we first make the following simple observation.
		\begin{claim}\label{claimtwovx}
			There are no edges of $G$ within the set ${B}\cup {D}$.
		\end{claim}
		\begin{poc}
			Suppose not, so that there is an edge $uv\in E(G)$ with $u,v\in
			{B}\cup {D}$. Note that $uv\not\in M$ since every edge of $M$ meets $Z\cup C$.
			
			Appealing to Claim~\ref{claimpath}, we first find an integer $0\leq s_1\leq r$ and a path
			$P_1$ in $G[{B}\cup {C}\cup {D}\setminus\{v,N_M(v)\}]$ with length $2s_1$
			from $u$ into ${B}$ which contains $s_1$ edges in $M$. As in the proof of Claim~\ref{claimonevx},
			we have $N_M(V(P_1))\subset V(P_1)$, and consequently $v,N_M(v)\not\in N_M(P_1)$
			
			Using Claim~\ref{claimpath} again, we find another integer $0\leq s_2\leq r$ and a path $P_2$ in $G[{B}\cup {C}\cup {D}\setminus V(P_1)]$ with length $2s_2$ from $v$
			into ${B}$ which contains $s_2$ edges in $M$.
			
			Switching edges into the matching $M$ along $P_1$ and $P_2$ creates a matching in $G$ with the same number of edges as $M$ which does not
			contain $u$ or $v$. Adding $uv$ creates a matching larger than $M$ in $G$, which is a contradiction.
		\end{poc}
		
		We know that there are at most $\e n/2$ vertices in ${Z}$ with at least $\e n/2$ neighbours in $G$ in ${B}\cup {D}$ (for otherwise, the iterative process we used above
		would have continued). We know by Claim~\ref{claimtwovx} that the set $B \cup D$ induces no edges, so we deduce that
		\[
		|E({B}\cup {D}, {Z}\cup {B}\cup {D})| \leq (\e n/2) |{Z}|+(\e
		n/2)|{B}\cup {D}|\leq \e n^2,
		\]
		as required.
	\end{proof}
	
	We are now in a position to prove the `almost covering' lemma that we require.
	
	\begin{lemma}\label{almost-cover}
		Let $1/n \ll \eta\ll\e \ll  \mu\ll 1$, and suppose that $\HH$ is an $(\e,\mu)$-coloured $3$-graph on $n$ vertices. Then we can find a collection of at most $\eta n$
		vertex-disjoint green spheres in $\HH$ which cover at least $(1-\mu^3)n$ vertices of $\HH$.
	\end{lemma}
	\begin{proof}Our primary strategy is to set aside a small set $T$ of vertices which will be used as apexes of double pyramids covering most of the vertices of $\HH$; in fact, we shall find it convenient to divide $T$ into three parts to be used at different stages of the proof. We also set aside a larger set $A$ of vertices that lie in few red edges; these vertices will be used to create additional spheres that cover vertices that are difficult to cover using double pyramids. Having set aside these sets, we define an auxiliary graph on the remaining vertices whose edges are those pairs which form green edges of $\HH$ with a reasonable number of vertices in $T$. We then use the regularity lemma to partition this graph. Regularity properties will then ensure that we may find double pyramids with apexes in $T$ covering almost all of any given regular pair of reasonable density. Applying Lemma~\ref{matchpart} to the reduced graph, we shall find disjoint regular pairs covering all but a few `bad' parts, and the additional conditions imposed by Lemma~\ref{matchpart} on these bad parts will allow us to use Lemma~\ref{mopbad} to find spheres covering most of the vertices from these bad parts with the help of additional `good' vertices, typically from $A$. We make this outline precise below.
		
		Let $k$ and $l$ be integers such that $ 1/n\ll 1/k\ll 1/l\ll \e$. Here $l$ will be the number of parts of a regularity partition, and $k$ will correspond to the size of spheres used in Lemma~\ref{mopbad}.
		
		Pick random disjoint sets $T_1,T_2,T_3\subset V(\HH)$ by
		including each vertex independently at random in $T_i$ with probability $\mu^3/10$ for each $i\in[3]$, and in none of them
		with the remaining probability $1-3\mu^3/10$. Observe that, by a simple application of
		Lemma~\ref{chernoff}, the properties~\ref{Tprop1} and
		\ref{Tprop2} below hold with high probability for such random choices; thus, we can find sets
		$T_1,T_2,T_3\subset V(\HH)$ satisfying the following conditions.
		
		\begin{enumerate}[label = {\bfseries{A\arabic{enumi}}}]
			\item For each $i\in [3]$, $\mu^3n/20\leq |T_i|\leq \mu^3n/8$.\label{Tprop1}
			\item For each $i\in [3]$, every pair of vertices contained in at least $\e^2 n$
			green edges of $\HH$ is in at least $\e^3 n$ green edges with a third
			vertex in $T_i$.\label{Tprop2}
		\end{enumerate}
		
		Let $A$ be the set of
		vertices in $V(\HH)\setminus (T_1\cup T_2\cup T_3)$ that are in at most $\e n^2$ red edges. Let $G$ be the graph with vertex set $V(\HH)\setminus (A\cup T_1\cup
		T_2\cup T_3)$ where for each pair of vertices $x,y\in V(G)$, we have $xy\in E(G)$ if and only if for each $i\in[3]$, there are at least $\e^3n$ green edges in $\HH$ of the form
		$xyt$ with $t\in T_i$. In particular, by~\ref{Tprop2}, we have $xy\in E(G)$ for every pair $x,y\in V(G)$ contained in at least $\e^2 n$  green edges of $\HH$.
		
		By Theorem~\ref{sz-reg}, we can find a partition $V(G) = Y_0\cup Y_1\cup \dots \cup Y_l$ so that $|Y_0|\le \e^2 n/100$, $|Y_1|= |Y_2| = \dots=|Y_l|$, and all but at most $\e^2 l^2/100$ pairs of these sets are $(\e^2/100)$-regular in $G$. Let $R$ be the reduced graph on the vertex set $\{Y_1,Y_2, \dots, Y_l\}$ with edges corresponding to regular pairs with density at least $\e^2/10$ in $G$. Using Lemma~\ref{matchpart}, we obtain, for some $m$, a partition of $V(R)$ into sets $\mathcal{Z}$, $\mathcal{B}$, $\mathcal{C}=\{C_1, C_2, \dots,C_m\}$ and
		$\mathcal{D}=\{D_1, D_2, \dots,D_m\}$ so that
		\begin{enumerate}
			\item $R[\mathcal{Z}]$ has a perfect matching,
			\item $C_iD_i\in E(R)$ for each $i\in [m]$, and
			\item $|E_R(\mathcal{B}\cup \mathcal{D}, \mathcal{Z}\cup \mathcal{B}\cup
			\mathcal{D})|\leq \e^2 l^2/10$.
		\end{enumerate}
		
		Now, let $Z$, $B$, $C$  and $D$ be, respectively, the vertices of $G$ contained in the classes in $\mathcal{Z}$, $\mathcal{B}$, $\mathcal{C}$ and $\mathcal{D}$. Note that the sets $T_1$, $T_2$, $T_3$, $A$, $B$, $C$, $D$, $Z$ and $Y_0$ partition $V(\HH)$. We begin with some simple properties of these sets.
		
		\begin{claim}\label{mildred}
			There are at most $\e n$ vertices $u\in B\cup D$ for which there are more than $\e n^2$ green
			edges of $\HH$ containing $u$ and at least one other vertex in $Z\cup B\cup D$.
		\end{claim}
		\begin{poc}
			Suppose to the contrary that there are more than $\e n$ such vertices.
			For each such vertex $u$, let $X_u$ be the set of vertices $v\in Z\cup B\cup D$ for which at least
			$\e n/2$ green edges of $\HH$ contain $u$ and $v$. We must have $\abs{X_u}\ge\e n/2$, since otherwise
			fewer than $\e n^2/2$ green edges contain $u$ and some vertex in $X_u$, and fewer than $\e n^2/2$
			contain $u$ but no vertex in $X_u$. Every ordered pair of the form $(u,v)$ with $v\in X_u$ corresponds to an edge of
			$G$, since there are more than $\e^2n^2$ green edges containing $\{u,v\}$, and $u,v\in Z\cup B\cup D\subset V(G)$.
			There are at least $\sum_u\abs{X_u}\geq \e^2 n/2$ such ordered pairs, yielding
			\[
			|E_G(B\cup D,Z\cup B\cup D)|\geq \e^2n^2/4.
			\]
			
			Now, $G$ has at most $n^2/l^2$ edges within $Y_i$ for each $i$, at most $\e^2 n^2/100$ edges between irregular pairs, and at most
			$\e^2 n^2/10$ edges between pairs of density below $\e^2/10$. Thus, all but at most $\e^2 n^2/8$ edges of $G$ lie between regular pairs corresponding to an edge in $R$, so we also have
			\[
			|E_G(B\cup D,Z\cup B\cup D)|<\left(\frac{n}{l}\right)^2
			|E_R(\mathcal{B}\cup \mathcal{D}, \mathcal{Z}\cup \mathcal{B}\cup
			\mathcal{D})|+ \frac{\e^2 n^2}{8}< \frac{\e^2n^2}{4},
			\]
			contradicting the bound above and establishing the claim.
		\end{poc}
		
		\begin{claim}\label{usefulclaim}
			For all but at most $\mu^3 n/64$ vertices $v\in B\cup D$, the green link graph
			$GL(v)$ has at least $(1/3+\mu/4)n$ vertices in $A\cup C$ with at least
			$(1/3+\mu /4)n$ neighbours in $A\cup C$.
		\end{claim}
		\begin{poc}
			By Claim~\ref{mildred} and Lemma~\ref{newlink}, there is a set $U\subset
			B\cup D$ of size at least \[|B\cup D|-\e n-\mu^4n>|B\cup D|-\mu^3 n/64\] such that
			for each $v\in U$, the green link graph $GL(v)$ has the following properties.
			\begin{enumerate}[label = {\bfseries B\arabic{enumi}}]
				\item The number of edges of $GL(v)$ with a vertex in $Z\cup B\cup D$ is at most $\e n^2$. \label{ethel1}
				\item At least $(1/3+\mu/2)n$ vertices of $GL(v)$ have degree at least $(1/3+\mu/2)n$. \label{ethel2}
			\end{enumerate}
			Fixing $v\in U$, we shall show the property in the claim holds for this vertex $v$,
			thereby completing the proof of the claim.
			
			Write $W_v$ for the set of vertices with degree at least $(1/3+\mu/2)n$ in $GL(v)$.
			Let $W_1=W_v\cap(Z\cup B\cup D)$. There are at least $|W_1| (1/3+\mu/2)n/2$ edges in $GL(v)$
			with a vertex in $Z\cup B\cup D$, so by~\ref{ethel1}, we have $|W_1|<6\e n<\mu^2 n$.
			
			Let $W_2\subset A\cup C$ be the set of vertices which, in $GL(v)$, have more than $\mu n/6$
			neighbours in $Z\cup B\cup D$. There are at least $|W_2| \mu n/6$ edges
			in $GL(v)$ with a vertex in $Z\cup B\cup D$, so by~\ref{ethel1}, we again have $|W_2|<\mu^2 n$.
			
			Every vertex in $W_v\setminus (W_1\cup W_2\cup T_1\cup T_2\cup T_3)$ is in $A\cup C$ and has at
			least $(1/3+\mu /3)n$ neighbours in $A\cup C\cup T_1\cup T_2\cup T_3$, and consequently has more
			than $(1/3+\mu /4)n$ neighbours in $A\cup C$. Since
			\[|W_v\setminus(W_1\cup W_2\cup T_1\cup T_2\cup T_3)|>|W_v|-5\mu^2 n>(1/3+\mu/4) n,\]
			the required property in the claim holds for $v$; this completes the proof.
		\end{poc}
		
		We are now in a position to start constructing the `almost spanning' collection of spheres that we require. Let $\mathfrak{S}_0$ be a collection of vertex-disjoint green spheres in
		$\HH[A\cup B\cup C\cup D]$ so that $|V(\mathfrak{S}_0)|$ is maximal, subject to the following conditions.
		\begin{enumerate}[label = {\bfseries C\arabic{enumi}}]
			\item Each sphere in $\mathfrak{S}_0$ has $4k$ vertices in $B\cup D$ and $2k+2$
			vertices in $A\cup C$. \label{sphprop1}
			\item $|D_i\cap V(\mathfrak{S}_0)|\leq |C_i\cap V(\mathfrak{S}_0)|$ for each $i\in [m]$.\label{sphprop2}
		\end{enumerate}
		First, note that if $\mathfrak{S}_0\neq\varnothing$ then, by~\ref{sphprop1}, we have
		\begin{equation}\frac{|(A\cup C)\cap V(\mathfrak{S}_0)|}{|V(\mathfrak{S}_0)|} =\frac{(2k+2)|\mathfrak{S}_0|}{(6k+2)|\mathfrak{S}_0|}<\frac{1}{3}+\frac{1}{k}<\frac{1}{3}+\frac{\e}{3};\label{ratio}\end{equation}
		therefore, we have the following.
		\begin{equation}3|(A\cup C)\cap V(\mathfrak{S}_0)|<|V(\mathfrak{S}_0)|+\e n.\label{ineq1}\end{equation}
		Second, since $A\cup C$ and $B\cup D$ are disjoint sets in $V(\HH)$, we have the following.
		\begin{equation}
		|(B\cup D)\setminus V(\mathfrak{S}_0)|\leq n-|V(\mathfrak{S}_0)|-|(A\cup C)\setminus V(\mathfrak{S}_0)|.\label{ineq2}
		\end{equation}
		
		We now argue that~\eqref{ineq1} and~\eqref{ineq2}, along with the maximality of
		$\mathfrak{S}_0$, imply the following.
		
		\begin{claim}\label{keyclaim} $|(B\cup D)\setminus V(\mathfrak{S}_0)|\leq |(A\cup
			C)\setminus V(\mathfrak{S}_0)|+\mu^3 n/16$.
		\end{claim}
		\begin{poc}
			Assume for the sake of a contradiction that
			\begin{equation}\label{gertie0}
			|(B\cup D)\setminus V(\mathfrak{S}_0)|> |(A\cup C)\setminus V(\mathfrak{S}_0)|+\mu^3 n/16.
			\end{equation}
			We will distinguish between values of $i$ for which~\ref{sphprop2} holds with enough extra room that it cannot be the barrier to adding another sphere to $\mathfrak{S}_0$ and values of $i$ for which~\ref{sphprop2} is closer to equality. To this end, let $X\subset[m]$ be the set of values of $i$ for which $|D_i\cap V(\mathfrak{S}_0)|>|C_i\cap
			V(\mathfrak{S}_0)|-4k$, or equivalently for which $|D_i\setminus V(\mathfrak{S}_0)|<|C_i\setminus V(\mathfrak{S}_0)|+4k$, and let $\tilde D=(\cup_{i\not\in X}D_i)\setminus V(\mathfrak{S}_0)$.
			It follows from~\ref{sphprop2} that
			\begin{align}
			|\tilde D|&=|D\setminus V(\mathfrak{S}_0)|-\sum_{i\in X}|D_i\setminus V(\mathfrak{S}_0)| \nonumber \\
			&>|D\setminus V(\mathfrak{S}_0)|-\sum_{i\in X}(|C_i\setminus V(\mathfrak{S}_0)|+4k) \nonumber \\
			&\geq |D\setminus V(\mathfrak{S}_0)|- |C\setminus V(\mathfrak{S}_0)|-4mk \nonumber \\
			&>|D\setminus V(\mathfrak{S}_0)|- |C\setminus V(\mathfrak{S}_0)|-\e n.\label{gertie1}
			\end{align}
			Finally, let $\tilde B=B\setminus V(\mathfrak{S}_0)$ and note that by~\eqref{gertie0} and
			\eqref{gertie1}, we have
			\begin{align*}|\tilde B\cup \tilde D|&\geq|(B\cup D)\setminus V(\mathfrak{S}_0)|- |C\setminus V(\mathfrak{S}_0)|-\e n\\
			&>\mu^3n/16-\e n>\mu^3 n/32.\end{align*}
			
			Thus, by Claim~\ref{usefulclaim}, there are at least $\mu^3 n/64$ vertices
			$x\in \tilde B\cup \tilde D$ for which there are at least \[(1/3+\mu/4) n-|(A\cup C)\cap V(\mathfrak{S}_0)|\] vertices in
			$(A\cup C)\setminus V(\mathfrak{S}_0)$ with at least \[(1/3+\mu/4) n-|(A\cup C)\cap V(\mathfrak{S}_0)|\] neighbours in $(A\cup C)\setminus V(\mathfrak{S}_0)$ in the green link graph $GL(x)$. Since $|(A\cup C)\cap V(\mathfrak{S}_0)|<(1/3+\e/3)n$ by \eqref{ratio}, we must have 
			$|(A\cup C)\setminus V(\mathfrak{S}_0)|>\mu n/8$.
			
			By the maximality of $\mathfrak{S}_0$, there are no green spheres with exactly $4k$ vertices in $\tilde B\cup \tilde D$ and $2k+2$
			vertices in $(A\cup C)\setminus V(\mathfrak{S}_0)$, since such a sphere would be disjoint from all other spheres in $\mathfrak{S}_0$, and, by the choice of $\tilde D$,
			adding it to $\mathfrak{S}_0$ would not violate~\ref{sphprop2}. By Lemma~\ref{mopbad} applied to the green edges of $\HH$ meeting two vertices of $(A\cup C)\setminus V(\mathfrak{S}_0)$ and one vertex of $\tilde B\cup \tilde D$, with $N=\mu n/8$, $\beta=\mu^2/8$ and $\gamma=\mu/4$, we have
			\begin{align*}
			(1/3+\mu/4) n-|(A\cup C)\cap V(\mathfrak{S}_0)|&\leq
			(1/2+\mu/4)|(A\cup C)\setminus V(\mathfrak{S}_0)|\\
			&\leq|(A\cup C)\setminus V(\mathfrak{S}_0)|/2+\mu n/4.
			\end{align*}
			We therefore have the following property.
			\begin{equation}
			n-3|(A\cup C)\cap V(\mathfrak{S}_0)|\leq 3|(A\cup C)\setminus V(\mathfrak{S}_0)|/2.\label{ineq3}
			\end{equation}
				
			Adding~\eqref{ineq1},~\eqref{ineq2} and~\eqref{ineq3} together, and rearranging, we have
			\begin{align*}
			|(B\cup D)\setminus V(\mathfrak{S}_0)|&\leq |(A\cup C)\setminus V(\mathfrak{S}_0)|/2+\e n \\
			& \leq |(A\cup C)\setminus V(\mathfrak{S}_0)|+\mu^3 n/16,
			\end{align*}
			contradicting~\eqref{gertie0} and completing the proof of the claim.
		\end{poc}
		
		Now, for each $i\in [m]$, let $C_i'=C_i\setminus V(\mathfrak{S}_0)$ and choose a set
		$D_i'\subset D_i\setminus V(\mathfrak{S}_0)$ be a set of size $|C'_i|$, noting that such a choice is made possible by~\ref{sphprop2}. Let $A'=A\setminus V(\mathfrak{S}_0)$ and \[B'=(B\cup D)\setminus \left(\left(\bigcup_{i\in
			[m]}D_i'\right)\cup V(\mathfrak{S}_0)\right),\]  and note that by Claim~\ref{keyclaim}, we have
		\begin{align}|B'|&=|(B\cup D)\setminus V(\mathfrak{S}_0)|-\sum_{i\in[m]}|D'_i|\nonumber \\
		&\leq |(A\cup C)\setminus V(\mathfrak{S}_0)|-\sum_{i\in[m]}|C'_i|+\mu^3 n/16\nonumber \\
		&=|A'|+\mu^3 n/16.\label{avsb}
		\end{align}
		Furthermore, note that the sets
		\begin{equation}\label{finalpartition}
		T_1,T_2,T_3,Z,V(\mathfrak{S}_0),A', B',C_1',\dots,C_l',D_1',\dots,D_l'
		\end{equation}
		constitute a partition of $V(\HH)$.
		
		\begin{claim}\label{coverv}
			There is a set $\mathfrak{S}_1$ of at most $\eta n/4$ vertex-disjoint green
			spheres in $\HH$ with vertices in $T_1\cup Z\cup Y_0$ which cover all but at most
			$\mu^3n/4$ vertices in $T_1\cup Z\cup Y_0$.
		\end{claim}
		\begin{poc}
			Set $n'=|Y_1|=\dots=|Y_l|$ and note that we have \[(1-\e^2/100)n/l\leq n'\leq n/l.\] Since $R[\mathcal Z]$ has a perfect matching, there exist distinct integers $a_1, a_2, \dots, a_r$ and $b_1,b_2, \dots, b_r$ such that \[Z=\bigcup_{i=1}^r(Y_{a_i}\cup Y_{b_i}),\] where each pair $(Y_{a_i},Y_{b_i})$ is $(\e^2/100)$-regular with density at least
			$\e^2/10$. In particular, for any subsets $P\subset Y_{a_i}$ and $Q\subset Y_{b_i}$ with $|P|,|Q|\ge \e^2n'/100$, there are at least
			$9\e^2|P||Q|/100$ pairs $p\in P$, $q\in Q$ for which there are at least $\e^3 n$ vertices $t\in T_1$ so that $pqt$ is a green edge of $\HH$.
			Consequently, if $T_1'$ is any set obtained by removing at most $\e^4 n$ vertices from $T_1$, the induced $3$-partite $3$-graph of green edges of $\HH$ between
			$P$, $Q$ and $T_1'$ contains at least $\e^5|P||Q||T_1'|$ edges.
			
			For each $i\in [r]$ in turn, we use Lemma~\ref{doubpy} to find at most $\eta n'/2$ green spheres in $Y_{a_i}\cup Y_{b_i}\cup T_1'$, each with two vertices in $T_1'$, covering at
			least $(1-\e^2/100)n'$ vertices in each of $Y_{a_i}$ and $T_{b_i}$, where $T_1'\subset T_1$ is the dynamically-updated set of vertices not yet used;
			at each stage, note that we we have $|T_1\setminus T_1'|\le\eta n/2\le\e^4 n$.
			
			The iterative procedure above gives us a total of at most $\eta|Z|/4\leq \eta n/4$ spheres covering all but at most
			$\e^2|Z|/100$ vertices of $Z$. Since $|T_1|\le \mu^3n/8$ and $|Y_0|\leq\e^2n/100$, all but at most \[\mu^3n/8+\e^2n/50<\mu^3n/4\] vertices in
			$T_1\cup Z\cup Y_0$ are covered, as required.
		\end{poc}
		
		\begin{claim}
			There is a set $\mathfrak{S}_2$ of at most $\eta n/4$ vertex-disjoint green spheres with vertices in $T_2\cup (\cup_{i\in [m]}(C'_i\cup D_i'))$
			which cover all but at most $\mu^3 n/4$ vertices in $T_2\cup (\cup_{i\in
				[m]}(C'_i\cup D_i'))$.
		\end{claim}
		\begin{poc}
			As we did in the proof of Claim~\ref{coverv}, for each $i\in[m]$ in turn, we may find at most $\eta n'/2$ green spheres covering all but at least $(1-\e^2/100)n'$ vertices
			in each of $C'_i$ and $D'_i$, using only vertices from those sets and vertices from $T_2$ which have not previously been used.
			This covers all but at most \[\mu^3n/8+\e^2n/100<\mu^3n/4\] vertices in $T_2\cup (\cup_{i\in [m]}(C'_i\cup D_i'))$, using at most $\eta n/4$
			vertex-disjoint green spheres, establishing the claim.
		\end{poc}
		\begin{claim} There is a set $\mathfrak{S}_3$ of at most $\eta n/4$
			vertex-disjoint green spheres with vertices in $T_3\cup A'\cup B'$ which
			cover all but at most $\mu^3 n/4$ vertices in $T_3\cup A'\cup B'$.
		\end{claim}
		\begin{poc}
			Choose partitions $A'=A_1\cup A_2\cup A_3$ and $B'=B_1\cup B_2$ such that $|A_1|=|B_1|$, $|A_2|=|A_3|$ and $|B_2|\leq \mu^3n/16$, noting that this is made possible by~\eqref{avsb}.
			
			First we show that we may find at most $\eta n/8$ vertex-disjoint green spheres in $T_3\cup A_1\cup B_1$ covering all but at most $\mu^3n/64$ vertices in each of
			$A_1$ and $B_1$, with each sphere using two vertices from $T_3$. We may of course assume that $|A_1|=|B_1|\geq \mu^3n/64$ as there is nothing to prove otherwise. Suppose $A_1'\subset A_1$ and $B_1'\subset B_1$ each have size at least $\mu^4n$. By the definition of $A$, there are at most $\e n^3$ red edges of $\HH$ meeting $A_1'$,
			and since $\HH$ is $(\e,\mu)$-coloured, there are at most $\e n^3$ uncoloured edges in $\HH$. Hence, at most $6\e n^2$ pairs $(a,b)$ with $a\in A_1'$ and $b\in B_1'$ are in at
			least $n/3$ non-green edges of $\HH$, so at least $|A_1'||B_1'|-6\e n^2>|A_1'||B_1'|/2$ pairs are in at least $\e^3 n$ green edges with the third vertex in $T_3$.
			By Lemma~\ref{doubpy}, we can find the required spheres.
			
			Similarly, since we have $|T_3\setminus T'_3|\le\eta n/4<\e^4n$, where $T'_3$ is the set of unused vertices from $T_3$ after almost covering $A_1 \cup B_1$ as above, we may find at most $\eta n/8$ vertex-disjoint green spheres in $T'_3\cup A_2\cup A_3$ covering all but at most $\mu^3n/64$ vertices in each of $A_2$ and $A_3$. In total, these spheres cover all but
			at most \[\mu^3n/8+\mu^3n/16+4\mu^3n/64=\mu^3n/4\] vertices in $T_3\cup A'\cup B'$.
		\end{poc}
		Therefore, recalling the partition in~\eqref{finalpartition}, it is now clear that $\mathfrak{S}_0\cup
		\mathfrak{S}_1\cup \mathfrak{S}_2\cup \mathfrak{S}_3$ is a set of at most $\eta n$ vertex-disjoint
		green spheres covering all but at most $\mu^3 n$ vertices in $\HH$, as required.
	\end{proof}
	
	\subsection{Proof of the main result}
	
	We have now gathered all the ingredients we require to complete the proof of our main result.
	
	\begin{proof}[Proof of Theorem~\ref{mainthm}]
		By Theorem~\ref{classify}, if the surface $\SS$ we wish to find a spanning copy of is not homoeomorphic to $\mathbb S^2$, then it is homeomorphic either to a connected sum of $g$ tori or to a connected sum of $g$ projective planes for some integer $g\geq 1$; if $\SS$ is homeomorphic to $\mathbb S^2$, then we set $g=0$.
		
		Let $1/n\ll \eta \ll \al$, $1/k\ll \e\ll \mu\ll 1/g$ (with the convention that if $g=0$, then we replace $1/g$ by $1$), and suppose that $\HH$ is a $3$-graph on $n$ vertices with $\codeg(\HH) \ge n/3 + \mu n$. We construct a spanning copy of $\SS$ in $\HH$ as follows.
		
		\smallskip
		\noindent \textbf{Colour the edges of $\HH$.} By Lemma~\ref{colourthm}, for $1/n \ll \al, 1/k \ll \e \ll \mu$, we can colour some of the edges of $\HH$ red and green to get an $(\e,\mu)$-coloured $3$-graph in which any monochromatic disjoint pair of edges is $(\al,k)$-connectible.
		
		\smallskip
		\noindent \textbf{Construct reservoirs.} Next, we choose two disjoint subsets $R_1,R_2\subset V(\HH)$ with $|R_1|=|R_2|=\mu n/144$ and the following additional properties.
		
		\begin{enumerate}[label = {\bfseries{D\arabic{enumi}}}]
			\item\label{P1} For any pair $(e,f)$ of disjoint green edges of $\HH$, and any subset $F\subset R_1$ with $|F|\leq 4\eta k n$, there exists a set $A\subset R_1\setminus F$ with $|A|\leq k$ for which there is a sphere with vertex set $V(e)\cup V(f)\cup A$ containing $e$ and $f$.
			\item\label{P2} For each pair of vertices $x,y\in V(\HH)$, there are at least $(1/3+\mu/2)|R_2|$ vertices $u$ in $R_2$ such that $uxy\in E(\HH)$.
			\item\label{P3} There are at least $\mu |R_2|/8$ vertices in $R_2$ contained in fewer than $\e n^2$ red edges each.
		\end{enumerate}
		
		To show that such a choice of $R_1$ and $R_2$ is possible, we employ the probabilistic method: we pick two disjoint subsets $R'_1,R'_2\subset V(\HH)$ with $|R'_1|=|R'_2|=\mu n/144$ uniformly at random, and show that they satisfy each of the above properties with probability at least $3/4$.
		
		First, let us verify that that $R'_1$ satisfies~\ref{P1} with probability at least $3/4$. Let $e$ and $f$ be two arbitrary disjoint green edges. Since both these edges are green, $e$ and $f$ are $(\al,k)$-connectible, so by Lemma~\ref{disjoint}, for some $l$ with $1\leq l\leq k$, there are $r\ge \al n/l$ pairwise disjoint sets $A_1, A_2, \dots,A_r\subset V(\HH)$ for which there is a sphere with vertex set $V(e)\cup V(f)\cup A_i$ containing $e$ and $f$. Each of these sets independently has probability $(\mu/144)^l$ of being contained in $R'_1$, so Proposition~\ref{chernoff} implies that at least $\al(\mu/144)^ln/(2l)$ of the $A_i$ are contained in $R'_1$ with probability $1-o(n^{-6})$. In this case, for any $F\subset R'_1$ with $|F|\leq 4\eta k n$, since each vertex in $F$ is in at most one of the disjoint sets $A_i$, at least \[\al(\mu/144)^ln/(2l)-|F|>0\] of the $A_i$ lie entirely within $R'_1\setminus F$. Since there are $O(n^6)$ possible pairs $(e,f)$, it follows that~\ref{P1} holds with high enough probability.
		
		Next, we prove that $R'_2$ satisfies~\ref{P2} with probability at least $3/4$. For a given pair of vertices $x,y\in V(\HH)$, at least $(1/3+\mu)n$ vertices form an edge with $x$ and $y$ in $\HH$. Each such vertex independently has probability $|R'_2|/n$ of being in $R'_2$. Consequently, Proposition~\ref{chernoff} implies that at least $(1/3+\mu/2)|R'_2|$ such vertices end up in $R'_2$ with probability $1-o(n^{-2})$; it follows from the union bound that~\ref{P2} holds with high enough probability.
		
		Finally, we prove that $R'_2$ satisfies~\ref{P3} with probability at least $3/4$. By virtue of being $(\e,\mu)$-coloured, there are at least $\mu n/4$ vertices in fewer than $\e n^2$ red edges in $\HH$, and each such vertex ends up in $R'_2$ with probability $|R'_2|/n$, so Proposition~\ref{chernoff} again implies that at least $\mu|R'_2|/8$ such vertices are in $R'_2$ with very high probability, proving that~\ref{P3} also holds with high enough probability.
		
		Going forward, we fix these two disjoint sets $R_1$ and $R_2$ with the properties described above and set $R=R_1\cup R_2$.
		
		\smallskip
		\noindent \textbf{Build absorbers.} Appealing to Lemma~\ref{greenabsorb}, we find a collection $\mathfrak{S}_1$ of at most $\eta n$ vertex-disjoint green-tinged spheres disjoint from $R$, which can absorb any subset of $R$ and which have at most $\mu n/9$ vertices in total; let $U$ be the set of vertices belonging to spheres in $\mathfrak{S}_1$.
		
		\smallskip
		\noindent \textbf{Find an almost-spanning collection of spheres.} Let $\GG= \HH[V(\HH)\setminus (R\cup U)]$ with colours inherited from $\HH$. Note that because $|R\cup U|\leq\mu n/8$, we have $|V(\GG)| \geq (1-\mu/8)n$, so \[\codeg(\GG)\geq(1/3+7\mu/8)n>(1/3+\mu/2)|V(\GG)|.\] Additionally, $\GG$ has at most $\e n^3<2\e |V(\GG)|^3$ uncoloured edges and at most $\e n^4< 2\e|V(\GG)|^4$ pairs of touching red and green edges. Finally, $\GG$ has at least $\mu n/8$ vertices in fewer than $\e n^2<2\e |V(\GG)|^2$ red edges. Consequently $\GG$ is $(2\e,\mu/2)$-coloured; in particular it contains at least $\mu n^3/100$ green edges. If $\SS$ is not homeomorphic to $\mathbb S^2$, then we employ Theorem~\ref{kov} to obtain a collection $\mathfrak{G}$ of $g$ green copies of $\TT$ or $\PP$ in $\GG$ as appropriate, using at most $12g$ vertices; if $\SS$ is homeomorphic to $\mathbb S^2$, then we set $\mathfrak{G}=\emptyset$.
		
		Easily, the $3$-graph $\GG'$ obtained by removing vertices of surfaces in $\mathfrak{G}$ from $\GG$ is $(3\e,\mu/3)$-coloured, so by Lemma~\ref{almost-cover}, we can find a collection $\mathfrak{S}_2$ of at most $\eta n$ vertex-disjoint green spheres in $\GG'$ which cover all but at most $\mu^3 n$ of the vertices of $\GG'$. Let $B$ be the set of vertices in $V(\HH)\setminus (R\cup U)$ which are not in any surface in $\mathfrak{G}\cup\mathfrak{S}_2$, and write $W$ for $V(\mathfrak{G}\cup \mathfrak{S}_2)$.
		
		To summarise, we have now decomposed the vertices of $\HH$ into $R=R_1 \cup R_2, U=V(\mathfrak{S}_1), W=V(\mathfrak{G}\cup \mathfrak{S}_2)$ and $B$, as shown in Figure~\ref{decomp}.
		
		\begin{figure}\begin{center}
				\begin{tikzpicture}[scale = 0.25]
				
				\filldraw[shift={(-33,8)},fill=lightgray,very thick,draw=black] (0,0) rectangle (18,6);
				\draw[shift={(-33,8)},thick,draw=black] (9,0) -- (9,6);
				\node at (-28.25,11) {\large $R_1$};
				\node at (-19.5,11) {\large $R_2$};
				\draw[shift={(-33,8)},->,thick,dashed,draw=black] (9.5,-1) -- (9.5,-4);
				\draw[shift={(-33,8)},->,thick,dashed,draw=black] (8.5,-1) -- (8.5,-4);
				
				\draw[shift={(-33,-5)},thick,black,dotted] (9,0) circle (8);
				\node at (-23.75,-5) {\large $U$};
				\draw[shift={(-24+4.4,-5)}, thick,black] (0,0) circle (3);
				\draw[shift={(-24+4.4,-5)}, thick,green] (0,-3) .. controls (1.2,-0.5) .. (0,3);
				\draw[shift={(-24+4.4,-5)}, thick,green] (-3,0) .. controls (0.8,-0.5) .. (3,0);
				
				\draw[shift={(-24-2.5,-5+3.5)}, thick,black] (0,0) circle (3);
				\draw[shift={(-24-2.5,-5+3.5)}, thick,green] (0,-3) .. controls (1.2,-0.5) .. (0,3);
				\draw[shift={(-24-2.5,-5+3.5)}, thick,green] (-3,0) .. controls (0.8,-0.5) .. (3,0);
				
				\draw[shift={(-24-2.5,-5-3.5)}, thick,black] (0,0) circle (3);
				\draw[shift={(-24-2.5,-5-3.5)}, thick,green] (0,-3) .. controls (1.2,-0.5) .. (0,3);
				\draw[shift={(-24-2.5,-5-3.5)}, thick,green] (-3,0) .. controls (0.8,-0.5) .. (3,0);
				
				\filldraw[shift={(12,-10)},fill=lightgray,very thick,draw=black] (0,0) rectangle (6,18);
				\node at (15,-1.5) {\large $B$};
				
				\draw[shift={(-2,-1)},thick,black,dotted] (0,0) circle (11.5);
				\node at (-2,-1) {\large $W$};
				
				\draw[shift={(-2+2,-1+6.5)}, thick,green] (0,0) circle (3);
				\draw[shift={(-2+2,-1+6.5)}, thick,green] (0,-3) .. controls (1.2,-0.5) .. (0,3);
				\draw[shift={(-2+2,-1+6.5)}, thick,green] (-3,0) .. controls (0.8,-0.5) .. (3,0);
				
				\draw[shift={(-2+2,-1-6.5)}, thick,green] (0,0) circle (3);
				\draw[shift={(-2+2,-1-6.5)}, thick,green] (0,-3) .. controls (1.2,-0.5) .. (0,3);
				\draw[shift={(-2+2,-1-6.5)}, thick,green] (-3,0) .. controls (0.8,-0.5) .. (3,0);
				
				\draw[shift={(-2-5.5,-1+4)}, thick,green] (0,0) circle (3);
				\draw[shift={(-2-5.5,-1+4)}, thick,green] (0,-3) .. controls (1.2,-0.5) .. (0,3);
				\draw[shift={(-2-5.5,-1+4)}, thick,green] (-3,0) .. controls (0.8,-0.5) .. (3,0);
				
				\draw[shift={(-2-5.5,-1-4)}, thick,green] (0,0) circle (3);
				\draw[shift={(-2-5.5,-1-4)}, thick,green] (0,-3) .. controls (1.2,-0.5) .. (0,3);
				\draw[shift={(-2-5.5,-1-4)}, thick,green] (-3,0) .. controls (0.8,-0.5) .. (3,0);
				
				\draw[shift={(4.5,-1)},thick,green] (-3.5,0) .. controls (-3.5,2) and (-1.5,2.5) .. (0,2.5);
				\draw[shift={(4.5,-1)},thick,green,xscale=-1] (-3.5,0) .. controls (-3.5,2) and (-1.5,2.5) .. (0,2.5);
				\draw[shift={(4.5,-1)},thick,green,rotate=180] (-3.5,0) .. controls (-3.5,2) and (-1.5,2.5) .. (0,2.5);
				\draw[shift={(4.5,-1)},thick,green,yscale=-1] (-3.5,0) .. controls (-3.5,2) and (-1.5,2.5) .. (0,2.5);
				\draw[shift={(4.5,-1)},thick,green] (-2,.2) .. controls (-1.5,-0.3) and (-1,-0.5) .. (0,-.5) .. controls (1,-0.5) and (1.5,-0.3) .. (2,0.2);
				\draw[shift={(4.5,-1)},thick,green] (-1.75,0) .. controls (-1.5,0.3) and (-1,0.5) .. (0,.5) .. controls (1,0.5) and (1.5,0.3) .. (1.75,0);
				
				\end{tikzpicture}
			\end{center}\label{decomp}
			\caption{The situation after finding an almost-spanning collection of spheres; the arrows signify the ability of the spheres in $U$ to absorb any subset of $R$.}
		\end{figure}
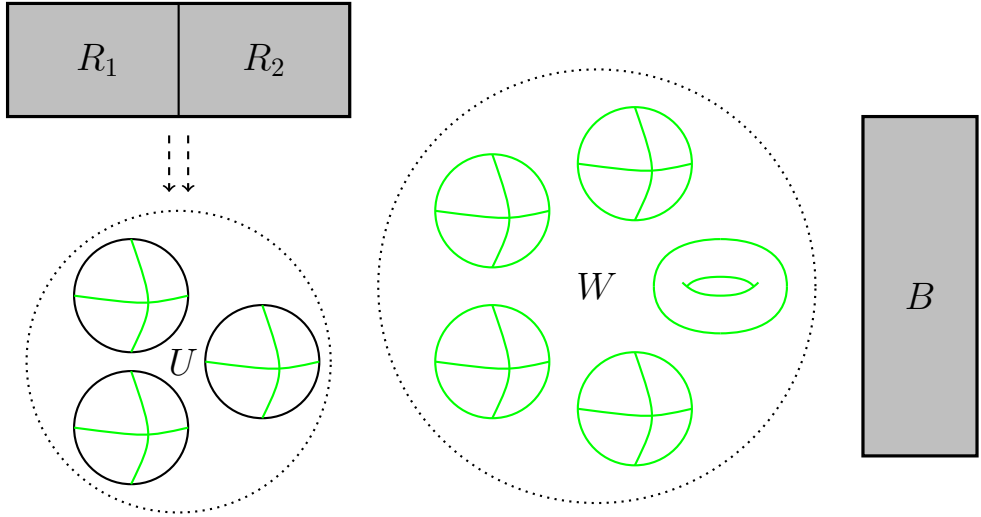
		
		\smallskip
		\noindent \textbf{Mop up uncovered vertices using the reservoir $R_2$.} Next, we show that we can mop up the set $B$ of uncovered vertices with green-tinged spheres using vertices in $B \cup R_2$ alone.
		\begin{claim}
			$\HH[B\cup R_2]$, with colours inherited from $\HH$, is $(144^4\e/\mu^4,\mu/3)$-coloured.
		\end{claim}
		\begin{poc}
			Indeed, this $3$-graph has $m$ vertices where $ m=|B\cup R_2| \ge |R_2|$, so
			\[\mu n/144\leq m<\mu n/144+\mu^2n/144^2.\] By~\ref{P2}, each pair of vertices in $B\cup R_2$ has codegree at least $(1/3+\mu/2)\mu n/144>(1/3+\mu/3)m$ in $\HH[B\cup R_2]$. Also, $\HH[B\cup R_2]$ certainly contains at most \[\e n^3\leq 144^3\e m^3/\mu^3< 144^4\e m^3/\mu^4\] uncoloured edges, and at most $\e n^4\leq144^4\e m^4/\mu^4$ pairs of differently-coloured touching edges. Finally,~\ref{P3} ensures that at least $\mu|R_2|/8>\mu m/12$ vertices are in fewer than $\e n^2\leq144^2\e m^2/\mu^2$ red edges.
		\end{poc}
		
		Consequently, it follows from Corollary~\ref{greensphere} that there is a collection $\mathfrak{S}_3$ of at most $\eta n$ vertex-disjoint green-tinged spheres in $\HH[B\cup R_2]$ which cover all the vertices in $B$.
		
		\smallskip
		\noindent \textbf{Connect into a single spanning surface using the reservoir $R_1$.}
		Order the surfaces in $\mathfrak{S}_1 \cup \mathfrak{S}_2 \cup \mathfrak{S}_3 \cup \mathfrak{G}$ as $S_1, S_2 \dots, S_l$ where $\mathfrak{S}_1=\{S_1, S_2 \dots,S_q\}$; since $\mathfrak{S}_i$ contains at most $\eta n$ spheres for each $1 \le i \le 3$, and $|\mathfrak{G}|\leq g$, we have $l\leq 4\eta n$.
		
		Since $V(\mathfrak{S}_1)=U$, $V(\mathfrak{G})$ and $V(\mathfrak{S}_2)$ are disjoint and $W = (V(\mathfrak{G}) \cup V(\mathfrak{S}_2)) \subset V(\HH)\setminus (R\cup U)$, and $V(\mathfrak{S}_3)\subset B\cup R_2$, where $B=V(\HH)\setminus (R\cup U\cup W)$, all the above surfaces are disjoint both from each other and from $R_1$. By the choice of $\mathfrak{S}_1$, for each sphere $S_i$ with $1\leq i\leq q$, there are two green edges $e_i,f_i\in S_i$ and a set of vertices $Q_i\subset R$ such that for any set $Q'\subset Q_i$, there is a sphere on the vertex set $V(S_i)\cup Q'$ containing $e_i$ and $f_i$, where the $Q_i$ are pairwise disjoint and $\cup_{i=1}^qQ_i=R$. Every sphere in $\mathfrak{S}_2\cup \mathfrak{S}_3$ is green-tinged, as is every surface in $ \mathfrak{G}$, so for each $i$ with $q<i\leq l$, we again choose two different green edges $e_i,f_i \in S_i$.
		
		For each $1 \le i \le l-1$ in turn, by~\ref{P1}, we find a sphere $\hat S_i$ containing $f_i$ and $e_{i+1}$ on the vertex set $V(f_i)\cup V(e_{i+1})\cup A_i$, where $|A_i|\leq k$ and $A_i\subset R_1\setminus F_i$, where $F_i$ is a nested sequence of sets given by $F_i=\cup_{j<i}A_j$; note that $|F_i|<ik<lk \leq 4\eta n k$, so this is indeed possible. The role of these spheres $\hat S_i$ will be to connect the $S_i$ together into a single surface. This works as follows: if we remove $f_i$ and $e_{i+1}$ from $\hat S_i$ as well as from $S_i$ and $S_{i+1}$ respectively, then the rest of  $\hat S_i$ yields a cylindrical tube that helps us form the connected sum of $S_i$ and $S_{i+1}$.
		
		However, in addition to gluing the $S_i$ together, we still need to take care of some uncovered vertices in $R$. Here, we rely on the properties of our absorbers, which allows us to incorporate them into the $S_i$ without affecting our gluing plans. Indeed, writing $F = \cup_{j<l}A_j$, let $R'_1=R_1\setminus F$ be the set of vertices in $R_1$ not used in any of the connecting spheres $\hat S_i$ for $1 \le i \le l-1$, and let $R'_2$ be the set of vertices in $R_2$ not used in any sphere in $\mathfrak{S}_3$, i.e., $R'_2=B\cup R_2\setminus V(\mathfrak{S}_3)$. For each $i\leq q$, let $Q'_i=Q_i\cap(R'_1\cup R_2')$, and let $S'_i$ be a sphere on vertex set $V(S_i)\cup Q'_i$ containing $e_i$ and $f_i$. For each $i>q$, set $S'_i=S_i$.
		
		To finish, we observe that the symmetric difference of the union of the edge-sets of all the $S'_i$ and all the $\hat S_i$, or in other words, the edge set
		\[\left(\bigcup_{j=1}^l S'_j\right)\cup\left(\bigcup_{j=1}^{l-1}\hat S_j\right)\setminus\{e_2,e_3,\dots,e_l,f_1, f_2, \dots,f_{l-1}\},\]
		yields a surface homeomorphic to $\SS$. The vertex set of this copy of $\SS$ is
		\[V' = V(\mathfrak{S}_1)\cup R'\cup V(\mathfrak{S}_2\cup\mathfrak{G})\cup V(\mathfrak{S}_3)\cup F.\]
		Finally, since
		\begin{align*}V(\mathfrak{S}_1)\cup R'&=U\cup R'_1\cup R'_2,\\
		V(\mathfrak{S}_2\cup\mathfrak{G})&=V(\HH)\setminus(R_1\cup R_2\cup U\cup B),\\
		V(\mathfrak{S}_3)&=B\cup R_2\setminus R'_2,\text{ and}\\
		F&=R_1\setminus R'_1,
		\end{align*}
		we see that $V'=V(\HH)$, so the copy of $\SS$ that we have found is spanning, proving the result.
	\end{proof}
	
	\section{Conclusion}\label{s:conc}
	In this paper, our main result, Theorem~\ref{mainthm}, asymptotically determines the minimum codegree guaranteeing the existence of a spanning copy of the sphere in a $3$-graph. Several natural questions remain; to begin with, it would be nice to eliminate the error term in our result, and in this direction, we conjecture the following.
	
	\begin{conjecture}
		Every $3$-graph $\HH$ on $n>3$ vertices with $\codeg(\HH) > n/3$ contains a spanning copy of the sphere $\mathbb{S}^2$.
	\end{conjecture}
	
	Just as natural is to ask what happens in `higher dimensions', or in other words, what the analogue of our main result for $r$-graphs ought to be for an arbitrary $r \in \N$. For any integer $r \ge 2$, the codegree of a set of $r-1$ vertices in an $r$-graph $\HH$ is the number of edges of $\HH$ containing the set in question, and writing $\delta_{r-1}(\HH)$ for the minimum codegree of an $r$-graph $\HH$, we conjecture the following.
	\begin{conjecture} \label{conj r}
		For each $r \ge 2$, any $r$-graph $\HH$ on $n>r$ vertices with $\delta_{r-1}(\HH) > n/r$ contains a spanning copy of the $(r-1)$-dimensional sphere $\ss^{r-1}$.
	\end{conjecture}
	That such a bound would be asymptotically best-possible is seen by the following construction generalising the one presented in Figure~\ref{figXYZ}. Given a positive integer $n$ divisible by $r$, let $X_1, X_2, \dots, X_r$, be $r$ disjoint sets of vertices of size $n/r$ each, and consider an $r$-graph $\HH$ on the vertex set $V = X_1 \cup X_2 \cup \dots \cup X_r$ constructed as follows. For a vertex $v\in V$, write $i(v)$ for the index such that $v\in X_{i(v)}$, and for a set $Y$ of $r$ vertices, let $i(Y) = \sum_{y\in Y} i(y) \imod r$; we then take the edge set of $\HH$ to consist precisely of those sets of $Y$ of $r$ vertices for which $i(Y) = 1$. Note that each `pattern' comprising $r$ choices among the $X_i$ (with repetitions allowed) gives rise to a distinct tight component in $\HH$, none of which are spanning because $\sum_{1\leq i \leq r} i \neq 1 \pmod r$; here, as before, we say that two edges in an $r$-graph touch if they meet in $r-1$ vertices, and a tight component is, once again, an equivalence class of the transitive closure of this relation. In particular, this $r$-graph $\HH$ does not contain a spanning copy of any closed manifold.
	
	In this paper, we addressed all two-dimensional surfaces simultaneously in Theorem~\ref{mainthm}. In the same spirit, one can also ask about other higher-dimensional manifolds than spheres, and in particular, whether there exist manifolds for which the codegree threshold differs qualitatively from that of the corresponding sphere; however, it is also perhaps worth remembering that unlike in the two-dimensional setting, not all higher-dimensional manifolds are triangulable; see~\citep{triang1, triang2} for example.
	
	Given that a minimum codegree of $n/3$ is the threshold at which an $n$-vertex $3$-graph is guaranteed to both have a spanning tight component and a spanning copy of the sphere, it is natural to wonder to what extent the main obstacle to finding a spanning copy of the sphere is the existence of a spanning tight component. In particular, one could ask whether the global codegree condition in the statement of Theorem~\ref{mainthm} of $\codeg(\HH) \geq (1/3+\mu)n$ for an $n$-vertex $3$-graph $\HH$ can be relaxed if we assume that $\HH$ has a spanning tight component $T$ and, say, that the codegrees of pairs inside $T$ are large. Perhaps surprisingly, this is not the the case! The following $3$-graph $\HH$ on $n$ vertices has minimum codegree almost $n/2$ in its unique tight component, and no spanning copy of any surface. The vertex set of $\HH$ consist of two designated vertices $u$ and $v$ and disjoint sets $X$ and $Y$ of size $(n-2)/2$ each, and the edge set of $\HH$ is obtained by starting with the complete $3$-graph on its vertex set and subsequently removing every edge meeting both $X$ and $Y$. Assuming $\HH$ contains a spanning copy of any surface $\SS$, remove any edges containing both $u$ and $v$ from $\SS$ to split it into two tight components; this easily leads to a contradiction by elementary topological arguments.
	
	The construction above, and its generalisation to $r$-graphs obtained by replacing the set $\{u,v\}$ by a set of $r-1$ vertices, leads to the following natural conjecture, made particularly attractive by the fact that it predicts a threshold at the codegree density of $1/2$, independent of the dimension.
	\begin{conjecture}\label{tightcompconj}
		For each $r \ge 2$, if $\HH$ is an $r$-graph on $n>r$ vertices comprising a single tight component such that, for every set $Z$ of $r-1$ vertices of $\HH$ that is contained in some edge of $\HH$, there are at least $n/2$ edges of $\HH$ containing $Z$, then $\HH$ contains a spanning copy of the $(r-1)$-dimensional sphere ${\mathbb S}^{r-1}$.
	\end{conjecture}
	
	Another natural question along the same lines is to ask how large a sphere can be found in a $3$-graph with a given minimum codegree below $n/3$. Here, it may be that every tight component meets relatively few vertices, but in that case pairs, within a given tight component will have large codegree, relative to the number of vertices met, so Conjecture~\ref{tightcompconj} may be helpful. However, it is difficult even to determine how large a tight component is guaranteed in this regime; for bounds in this direction, see \cite{tight-comps}.
	
	Let us close by remarking that many extremal results about cycles in graphs ought to have natural generalisations formulated in terms of triangulations of spheres in hypergraphs. We shall resist the temptation to list further open problems of this kind here, although we would certainly be very interested to see more results of this nature. Perhaps the main message of this paper is that it is possible to do `extremal simplicial topology' with a flavour similar to extremal graph theory, and that some of the major techniques in the latter field, like regularity and absorption, can be brought to bear on the former.
	
	\section*{Acknowledgements}
	The first and second authors were supported by the European Research Council grant 639046 under the European Union's Horizon 2020 research and innovation programme, and the fourth author was partially supported by NSF Grant DMS-1800521.
	
	Some of the research in this paper was carried out while the third and fourth authors were visiting the University of Warwick, and was continued while the first and second authors were visiting the University of Cambridge; we are grateful for the hospitality of both universities.
	
	We would also like to thank Tim Gowers, Richard Mycroft and Istv\'an Tomon for helpful discussions at various stages of this project.
	\bibliographystyle{amsplain}
	\bibliography{spanning_surfaces}
	
\end{document}